\documentclass[a4paper,12pt]{amsart}
\usepackage[ascii]{inputenc}
\usepackage{lmodern}
\usepackage{amsmath,amsfonts,amssymb,amsthm}
\usepackage{fullpage}
\usepackage[ocgcolorlinks, linkcolor=red]{hyperref}

\usepackage{url}
\usepackage{graphicx}
\usepackage{xcolor}
\usepackage{enumerate}

\usepackage[T1]{fontenc}
\rmfamily

\theoremstyle{plain}
\newtheorem{theorem}{Theorem}[section]
\newtheorem{lemma}[theorem]{Lemma}
\newtheorem{proposition}[theorem]{Proposition}

\theoremstyle{definition}
\newtheorem{remark}[theorem]{Remark}
\newtheorem{definition}[theorem]{Definition}
\newtheorem{example}[theorem]{Example}

\numberwithin{equation}{section}

\newcommand{\R}{\mathbb{R}}
\newcommand{\C}{\mathbb{C}}
\newcommand{\N}{\mathbb{N}}

\renewcommand\Re{\operatorname{Re}}
\renewcommand\Im{\operatorname{Im}}
\newcommand{\ov}[1]{\overline{#1}}
\DeclareMathOperator{\arcosh}{arcosh}
\DeclareMathOperator{\arsinh}{arsinh}

\newcommand\p{\partial}

\begin{document}

\title{Cones with convoluted geometry that always scatter or radiate}

\author{Emilia L.K. Bl{\aa}sten}\address{Emilia L.K. Bl{\aa}sten,
  Department of Mathematics and Systems Analysis, Aalto University,
  FI-00076 Aalto, Finland \textnormal{and} Division of Mathematics, Tallinn
  University of Technology Department of Cybernetics, 19086 Tallinn,
  Estonia}\email{emilia.blasten@iki.fi}

\author{Valter Pohjola} 
\address{Valter Pohjola, BCAM -- Basque Center for Applied Mathematics, Bilbao, Spain}
\email{valter.pohjola@gmail.com}

\date{}

\begin{abstract}  
We investigate fixed energy scattering 
from conical potentials having an irregular cross-section. The
incident wave can be any arbitrary non-trivial Herglotz wave. We show
that a large number of such local conical scatterers scatter all
incident waves, meaning that the far-field will always be non-zero. In
essence there are no incident waves for which these potentials would
seem transparent at any given energy. We show more specifically that
there is a large collection of star-shaped cones whose local
geometries always produce a scattered wave. In fact, except for a
countable set, all cones from a family of deformations between a
circular and a star-shaped cone will always scatter any non-trivial
incident Herglotz wave. Our methods are based on the use of spherical
harmonics and a deformation argument. We also investigate the related
problem for sources. In particular if the support of the source is locally a thin cone,
with an arbitrary cross-section, then it will produce a non-zero far-field.
\end{abstract}

\maketitle

\tableofcontents

\section{Introduction}

\noindent
When a wave of constant wavenumber encounters a penetrable obstruction, it typically
produces a scattered wave. The scattered wave transmits information about the obstruction,
which is a fact that is exploited by various imaging modalities \cite{colton96_simpl_method_solvin_inver_scatt, kirsch98_charac_shape_scatt_obstac_using}.
However, an incident wave 
does not necessarily scatter even in the presence of a scatterer \cite{cakoni21_trans_eigen}.
In the quantum mechanical context we call the energy of such a wave function 
a \emph{non-scattering energy}. When dealing with 
with acoustic or electromagnetic scattering, we use the term \emph{non-scattering wavenumber} instead.
It might happen that a potential or an obstruction
is effectively transparent at certain fixed energies or wavenumbers, producing
a zero far-field irrespective of the incident wave (see e.g.
\cite{grinevich95_trans_poten_at_fixed_energ_dimen_two,newton62_const_poten_from_phase_shift,regge59_introd_to_compl_orbit_momen}).
A striking converse to this type of phenomenon was first studied in \cite{blaasten14_corner_alway_scatt}, where the authors 
showed that an obstacle with a corner always scatters, i.e. it always produces
a scattered wave that does not decay quickly, regardless of what the incident wave is or what energy it has.
This type of obstacle is not transparent to any incident wave, i.e. it is always visible. 
The absence of non-scattering energies has in recent years been studied by a number of authors, see e.g.
\cite{blaasten17_vanis_near_corner_trans_eigen,blaasten17_adden_to,blaasten21_scatt_by_curvat_radiat_sourc,blaasten14_corner_alway_scatt,cakoni21_singul_almos_alway_scatt,cakoni21_corner_scatt_operat_diver_form,elschner15_corner_edges_alway_scatt,elschner17_acous_scatt_from_corner_edges_circul_cones,hu16_shape_ident_inver_medium_scatt,paeivaerinta17_stric_convex_corner_scatt}.
One of the themes of this paper is to study the absence of non-scattering-energies for potentials that
are obtained from irregular cones. 

It should also be pointed out that non-scattering energies are closely related to the so called 
\emph{interior transmission eigenvalues} \cite{cakoni13_trans_eigen_inver_scatt_theor}. These are, in the acoustic setting, eigenvalues for a certain
non elliptic and non self-adjoint spectral problem on the support of an inhomogenity of the refractive index.
A non-scattering-energy is a always an interior transmission eigenvalue when the
scatterer has compact support. The converse does not hold. It is well known that transmission eigenfunctions can be approximated by normalized Herglotz waves \cite{weck03_approx_by_hergl_wave_funct}. Such waves will produce arbitrarily small scattering. However this does not imply that there would be an incident wave producing \emph{no scattering}. This only happens in the case of a non-scattering energy, and only if this sequence of Herglotz waves aproximates a non-scattering incident wave.
The study of interior transmission eigenvalues has it roots in 
the analysis of certain numerical reconstruction methods in inverse scattering, namely
the linear sampling method and the factorization method (for more on these see \cite{colton96_simpl_method_solvin_inver_scatt,colton19_inver_acous_elect_scatt_theor,kirsch98_charac_shape_scatt_obstac_using,kirsch08_factor_method_inver_probl}).
Energies which admit interior transmission eigenvalues cause challenges for these methods.
For a survey on the topic of interior transmission eigenvalues see \cite{cakoni13_trans_eigen_inver_scatt_theor}.

Another topic related to non-scattering are sources that produce no wave in the far field. A source $f$ in acoustic
scattering $(\Delta+k^2)u=f$ produces a far-field $u^\infty(\theta) = C \widehat{f}(k\theta)$, for details, see
\cite{blaasten21_scatt_by_curvat_radiat_sourc}. Therefore understanding source
scattering at a fixed wavenumber is related to understanding the Fourier
restriction problem \cite{grafakos04_class_fourier}.
By a nonradiating source, we mean a source that does not produce a wave in the far field
at some wavenumber or energy.
The topic of sources that radiate at all energies was first studied in \cite{blaasten18_nonrad_sourc_trans_eigen_vanis}, where it was shown that
that nonradiating sources having a convex or nonconvex corner or edge on their boundary must vanish there.
Subsequent studies of sources that always radiate include \cite{blaasten18_radiat_non_radiat_sourc_elast,blaasten21_scatt_by_curvat_radiat_sourc,blaasten21_elect_probl_corner_its_applic}.
Here we extend the results in \cite{blaasten18_nonrad_sourc_trans_eigen_vanis} by showing that sources obtained from very thin or very wide cones will always produce a radiating
wave irrespective of their cross section.

\medskip
\noindent
We will describe very briefly and in a non technical wave the various
concepts needed to understand our theorems. This is also to fix
notation. Proper mathematical definitions are given in
Section~\ref{sec_prelim}.

Recall that the scattering of an incident particle wave of energy
$\lambda>0$ by a potential $V$ can be modeled by the equation
\begin{equation} \label{eq_Schrodinger0}
  (-\Delta + V -\lambda) u = 0
\end{equation}
in $\R^n$. Here $u$ is the superposition of the wave function $u_i$ of
an incident particle and a scattered wave $u_s$, which is created by
the interaction between $u_i$ and the scatterer $V$. The scattered
wave will have an asymptotic expansion, and a detectable part of that
is called the scattering amplitude or far-field pattern.

By a non-scattering energy for a potential $V$, we mean a $\lambda>0$
for which there is an incident wave which produces no scattering from
$V$ at energy $\lambda$. In other words, the scattering amplitude is
zero even though the incident wave is not. A potential which
\emph{always scatters} is one that does not have any non-scattering
energies.

\medskip
\noindent
We will analyze potentials whose support has a conical singularity on
its boundary. We will consider cones $C \subset \R^n$, $n=3$, that are
determined by a compact cross section $K \subset \R^{n-1}$, $0 \in
\operatorname{int} K$ in suitable coordinates by
\begin{align}  \label{eq_cone}
C = \big\{  (tx',t) \in \R^n \;:\; x' \in K ,\,t \geq 0  \big\}.
\end{align}
For technical reasons, we will also need some regularity assumptions.
\begin{definition} \label{def_regcone}
We call a cone $C$ of the form \eqref{eq_cone} \emph{regular} if the
following conditions are satisfied.
\begin{enumerate}[(i)]
\item $C$ is contained in a strictly convex closed circular cone,
\item $C$ has a connected exterior, and
\item $C$ has a bounded cross-section $K\subset\mathbb R^2$ such that
  $\chi_K\in H^\tau(\mathbb R^2)$ for some $\tau\in (1/4,1/2)$.
\end{enumerate}
\end{definition}
\noindent
We will in particular be interested in star-shaped cones, which we
define as follows.
\begin{definition} \label{def_starCone0}
We say that a cone $C$ of the form \eqref{eq_cone} is
\emph{star-shaped} if there exists a continuous $\sigma \colon [0,2\pi] \to
(\rho_0,\pi/2)$, $\rho_0 \in (0,\pi/2)$ with $\sigma(0)=\sigma(2\pi)$
and
\begin{equation} \label{eq_starShaped}
  C \cap \mathbb S^2 := \big\{ (\sin\vartheta\cos\varphi,
  \sin\vartheta\sin\varphi,\cos\vartheta) \in \mathbb S^2 \;:\;
  \varphi \in [0,2\pi),\, 0\leq \vartheta < \sigma(\varphi) \big\}.
\end{equation}
Such a cone is denoted by $C^\sigma$ and we usually take both $\sigma$
and $\rho_0$ given by the context.
\end{definition}

\noindent
We will relate a conical potential to a cone as follows.
Given a cone $C$ with coordinates chosen as in \eqref{eq_cone}, let $V_C$ be the potential 
\begin{align}  \label{eq_VC}
V_C := \varphi\,\chi_C+\Phi,
\end{align}
where $\varphi\in C^{1/4+\varepsilon}_{\mathrm c}(\mathbb R^3)$, $\epsilon >0 $,
and $\Phi\in e^{-\gamma\left|\cdot\right|}\,L^2(\mathbb R^3)$, $\gamma > 0$, 
and $\operatorname{supp} \Phi \subset \{ x_3 > 0\}$.

Next, we state our main result for scattering from a potential. The theorem
is a direct consequence of Theorem~\ref{thm_admScatter} and
Proposition~\ref{prop_affine} (in the latter set $C^\sigma_1 = C^\sigma$,
to obtain the theorem).

\begin{theorem} \label{thm_starScat}
Assume that $C^\sigma$ is a star-shaped cone that is regular, in the sense of definition
\ref{def_regcone}. 
Then for all $\delta >0$, there exists a cone $C$, 
such that the Hausdorff distance\footnote{  
Here we use the \emph{Hausdorff distance} 
between the sets $A$ and $B$ which is given by
$$
\operatorname{dist}_H(A,B) := \max \Big( \sup_{x\in A}d(x,B), \,\sup_{y\in B}d(y,A) \Big).
$$
where $d(x,A) := \inf_{y \in A} d(x,y)$, and $d$ is induced by the Euclidean metric.
}
$$
\operatorname{dist}_H \big(\p C \cap S^2, \p C^\sigma \cap S^2 \big) < \delta,
$$
and such that all conical potentials $V_C$ of the form \eqref{eq_VC} always scatter.
\end{theorem}

\begin{remark}
  We remark that more can be said: If we deform $C^\sigma$
  continuously to a circular cone according to
  Definition~\ref{def_starShaped}, then all the cones in that family
  will have a shape that always scatters except possibly for a countable set of
  exceptions given by Proposition~\ref{prop_affine}. 
\end{remark}

\noindent
The main novelty of Theorem~\ref{thm_starScat} is that it shows that
there are a large number of conical potentials that always scatter,
which have an irregular geometry, and the scattering happens for
\emph{any} incident Herglotz wave. Earlier results in the case of
$\R^3$, have dealt with fairly regular geometries such as corners
between hyperplanes \cite{blaasten14_corner_alway_scatt} and circular
cones or curvilinear polyhedra
\cite{elschner15_corner_edges_alway_scatt,elschner17_acous_scatt_from_corner_edges_circul_cones,paeivaerinta17_stric_convex_corner_scatt}).
Where \cite{paeivaerinta17_stric_convex_corner_scatt} uses a similar
approach as we do here, utilizing complex geometrical optics
solutions, and
\cite{elschner15_corner_edges_alway_scatt,elschner17_acous_scatt_from_corner_edges_circul_cones}
use techniques from the theory of boundary value problems in corner
domains. More recent results apply to more complicated geometries,
such as \cite{cakoni21_singul_almos_alway_scatt} using stationary
phase method and \cite{salo21_free_bound_method_non_scatt_phenom}
using the methods from the theory of free boundary value
problems. These results are not applicable to a general Herglotz
incident wave and instead assume that it does not vanish on the
boundary. Non vanishing is a major technical simplification. The proof
in \cite{blaasten14_corner_alway_scatt} would be much shorter if one
would assume that $u_i(0) \neq 0$ there. This is also demonstrated in
this paper between Section~\ref{sec_source} and sections
\ref{sec_cgo}, \ref{sec_admissible}. The main difficulty is as
follows: if one assumes the source problem or a non vanishing incident
wave, then one only needs to prove the non vanishing of an integral
involving spherical harmonics of degree 2, as in
Definition~\ref{admissible-source-corner}. With a general incident
wave one needs instead to prove the non vanishing of the integral for
all degrees, as in Definition~\ref{def_admissible-medium-corner}.

\medskip
\noindent
Let us discuss briefly the proof of Theorem~\ref{thm_starScat}. Our
starting point is the use of complex geometrical optics solutions and
the theory of spherical harmonics. We use the latter to derive a
projection condition for cones which will guarantee potentials that
always scatter. We call such cones \emph{admissible medium cones}. See
Definition~\ref{def_admissible-medium-corner}.

The admissibility of a cone can be further verified by means of
certain determinants being non zero. We use this determinant condition
to prove that circular cones are admissible medium cones (and that
they therefore always scatter) by computing explicit formulas for the
determinants. This gives an alternative proof for the three
dimensional result of
\cite{elschner17_acous_scatt_from_corner_edges_circul_cones,paeivaerinta17_stric_convex_corner_scatt}
without the issue of a countable set of unhandled cones.

Testing for the determinant condition requires us to use certain results on
the associated Legendre polynomials. We derive them in
Section~\ref{sec_assocLegendre}. In particular, we need a modification
of the classical Christoffel-Darboux formula in page 43 of
\cite{szego75_orthog_polyn}. Then to analyze potentials related to star-shaped cones,
we perform a deformation argument, where we interpolate between the points of 
a circular cone and a star-shaped cone. Spherical harmonics are analytic, which
can be used to show that the determinants of the
deformed cones depend analytically on the deformation
parameter. Analyticity then guarantees that certain critical integrals
cannot vanish except in a countable set of points, and Theorem~\ref{thm_starScat} follows.

We would like to further point out that even though our construction
of complex geometrical optics solutions follows the argument of
\cite{paeivaerinta17_stric_convex_corner_scatt}, we need to modify
their argument to work with cones with irregular cross section. This
is done in Section~\ref{sec_cgo}, where we also provide some explicit
examples of cones with rather complex geometry to which our argument
applies but previous arguments didn't.

\medskip
\noindent
We turn to the related problem of sources that radiate at every energy
level next. Several types of source problems in wave propagation can
be modeled by the Helmholtz or static Schr\"odinger equation. We will
consider the following,
\begin{align}  \label{eq_source_prob}
  (-\Delta - \lambda) u = f
\end{align}
in $\R^n$ together with a radiation condition at infinity, the
Sommerfeld radiation condition \eqref{eq_SRC}. It selects the outgoing
wave from all possible solutions to \eqref{eq_source_prob}.
The above problem models a wave $u$ created by the source $f$ that is
oscillating at energy $\lambda>0$ and is radiating out to infinity.
We will consider sources of the form
\begin{align} \label{eq_source_term}
f = \varphi \chi_C + \Phi,
\end{align}
where $\varphi \in L^\infty(\R^n)$ is compactly supported, $\Phi\in
e^{-\gamma\left|\cdot\right|}\,L^2(\mathbb R^n)$ with $\gamma > 0$ is
supported on a half-space, and $C$ is a cone of the form
\eqref{eq_cone} with vertex outside the above-mentioned half-space. It
can be shown that the wave $u$ has the asymptotic expansion
\begin{equation}\label{eq_source_asympt}
  u(x) = \frac{e^{i\sqrt{\lambda}|x|}}{|x|^\frac{n-1}{2}}
  \alpha_r(\theta) + O\big(|x|^{-\frac{n+1}{2}} \big).
\end{equation}
The function $\alpha_r\in L^2(\mathbb S^{n-1})$ is called the
\emph{far-field pattern} radiated by $f$. We say that a source term
$f$ always radiates if $\alpha_r \neq 0$ for all $\lambda > 0$.

We will be interested in sources $f$ determined by a fairly general
family of cones that we call \emph{admissible source cones},
Definition~\ref{admissible-source-corner}. These cones can have very
irregular cross sections, e.g. with a fractal boundary. Our main
result concerning these types of sources is a direct consequence of
Theorem~\ref{source-theorem} in Section~\ref{sec_source}.

\begin{theorem} \label{thm_source}
Assume that $C\subset\mathbb R^3$ is a cone of the form
\eqref{eq_cone}, that is also an admissible source cone in the sense
of Definition~\eqref{admissible-source-corner}.  Suppose furthermore
that $f$ is the source term \eqref{eq_source_term}, where $\varphi$ is
H\"older continuous at $0$ and $\varphi(0) \neq 0$. Then for all
$\lambda > 0$, we have that $\alpha_r \neq 0$ in
\eqref{eq_source_asympt}. That is, $f$ always radiates.
\end{theorem}

\noindent
The first results on sources that always radiate were obtained in
\cite{blaasten18_nonrad_sourc_trans_eigen_vanis}, where it was shown
that sources having a convex or non-convex corner or edge on their
boundary always radiate. The main novelty of Theorem~\ref{thm_source}
is that it shows that a source determined by a cone with a singularity
at the vertex, will always radiate essentially regardless of the
geometry of the cross section of the conical source. This follows
because of propositions \ref{prop_small_angles} and
\ref{prop_hollow_source}.

\bigskip
\noindent
The paper is structured as follows. In Section~\ref{sec_prelim} we
review some preliminary results and notation that we use in the other
sections. In particular, we set the notation for spherical
coordinates. We then study the source problem in
Section~\ref{sec_source}. After this, we construct complex geometrical
optics solutions for potentials constructed from irregular cones in
Section~\ref{sec_cgo}. In Section~\ref{sec_admissible}, we define the
concept of an admissible medium cone and show that potentials related
to these always scatter. Section~\ref{sec_circular} proves that
circular cones are admissible medium cones. After this, in
Section~\ref{sec_density}, we prove
Theorem~\ref{thm_starScat}. Finally, in
Section~\ref{sec_assocLegendre}, we derive various formulas for
associated Legendre polynomials that were used in earlier sections.

\section{Preliminaries} \label{sec_prelim}

\noindent
We will review various results that are needed in the sequel. We start
by specifying some function spaces after which we give a short review
of basic scattering theory based on
\cite{agmon76_asymp_proper_solut_differ_equat,hoermander05_analy_linear_partial_differ_operat_ii}.
It is the same setup as in
\cite{paeivaerinta10_inver_scatt_magnet_schroe_operat,paeivaerinta17_stric_convex_corner_scatt}
for static scattering theory. The reader can find further details in
these references.

Following
\cite{agmon76_asymp_proper_solut_differ_equat,hoermander05_analy_linear_partial_differ_operat_ii,paeivaerinta10_inver_scatt_magnet_schroe_operat,paeivaerinta17_stric_convex_corner_scatt}
we use the spaces $B(\R^n)$ and $B^*(\R^n)$.  The space $B(\R^n)$
consists of those $u \in L^2(\R^n)$ for which the norm
\begin{equation*}
  \|u\|_{B(\R^n)} = \sum_{j=1}^{\infty} ( 2^{j-1} \int_{X_j} |u|^2
  \,dx )^{1/2}
\end{equation*}
is finite. See Section~14.1 in
\cite{hoermander05_analy_linear_partial_differ_operat_ii}. Here
$X_1 = \{ |x| < 1 \}$ and $X_j = \{ 2^{j-2} < |x| < 2^{j-1} \}$ for $j
\geq 2$. This is a Banach space whose dual $B^*(\R^n)$ consists of
all $u \in L^2_{\text{loc}}(\R^n)$ such that
\begin{equation*}
  \|u\|_{B^*(\R^n)} = \sup_{R > 1} \left[ \frac{1}{R} \int_{|x| < R}
    |u|^2 \,dx \right]^{1/2} < \infty.
\end{equation*}
The set $C^{\infty}_c(\R^n)$ of compactly supported smooth functions
is dense in $B(\R^n)$ but not in $B^*(\R^n)$. Their closure in
$B^*(\R^n)$ is denoted by $\mathring{B}^*(\R^n)$, and $u \in
B^*(\R^n)$ belongs to $\mathring{B}^*(\R^n)$ if and only if
\begin{equation*}
  \lim_{R \to \infty} \frac{1}{R} \int_{|x| < R} |u|^2 \,dx = 0.
\end{equation*}
We will also use $B^*_2$ and $\mathring{B}_2^*$, which take into
account the derivatives up to second order of the function. These are
defined via the norm
\begin{equation*}
	\|u\|_{B^*_2(\R^n)} = \sum_{|\alpha| \leq 2} \|D^{\alpha} u\|_{B^*(\R^n)}
	= \|u\|_{\mathring{B}^*_2(\R^n)}. 
\end{equation*}

Next, we give a short review of the relevant parts of scattering
theory. In static or time harmonic scattering theory of energy
$\lambda>0$ one considers the equation
\begin{equation} \label{eq_Schrodinger}
  (-\Delta + V - \lambda) u = 0
\end{equation}
in $\R^n$ where $u$ is a total wave function that is the superposition
of an unperturbed incident wave function $u_i$ and a wave $u_s$
scattered by the potential function. That is, $u = u_i + u_s$ where
\begin{equation} \label{eq_scat_parts}
  (-\Delta - \lambda ) u_i = 0, \qquad 
  (-\Delta + V -\lambda) u_s = V u_i
\end{equation}
in $\R^n$. One further requires that the scattered wave $u_s$
satisfies the Sommerfeld radiation condition
\begin{equation}\label{eq_SRC}
  \lim_{r \to \infty} r^{\frac{n-1}{2}} \big( \p_r u_s - iku_s \big) = 0
\end{equation}
uniformly over $\hat{x} = x/r$ where $r=|x|$.

The potential $V$ is a \emph{short range} potential, by which we mean
that $V \in L^\infty(\R^n)$ and that there are constant $C,\epsilon >
0$, such that
\begin{equation} \label{eq_V_short_range}
  |V(x)| \leq C (1+|x|)^{-1-\epsilon}
\end{equation}
almost everywhere in $\R^n$. Note that this allows for potentials
whose support is unbounded.

Will will furthermore assume that the incident wave is given by a
Herglotz function, i.e.
\begin{equation} \label{eq_Herglotz_ui}
  u_i(x) = \int_{\mathbb S^{n-1}} e^{i\sqrt{\lambda}\theta \cdot x} g(\theta)
  \,d S, \qquad g \in L^2(\mathbb S^{n-1}).
\end{equation}
It follows that $u_i \in B^*_2(\R^n)$. See Proposition~2.1 in \cite{paeivaerinta10_inver_scatt_magnet_schroe_operat}. They also
show that the incident particle $u_i$ has the asymptotics
\begin{equation} \label{eq_ui_asymp}
  u_i(x) = \frac{e^{i\sqrt{\lambda}|x|}}{|x|^\frac{n-1}{2}} g(\theta)
  + \frac{e^{-i\sqrt{\lambda}|x|}}{|x|^\frac{n-1}{2}} g(-\theta) +
  O\big(|x|^{-\frac{n+1}{2}} \big)
\end{equation}
and moreover that the scattered wave has the asymptotics
\begin{equation} \label{eq_us_asymp}
  u_s(x) = \frac{e^{i\sqrt{\lambda}|x|}}{|x|^\frac{n-1}{2}}
  \alpha_s(\theta) + O\big(|x|^{-\frac{n+1}{2}} \big).
\end{equation}
The function $\alpha_s$ is called the \emph{scattering amplitude} or
the \emph{far field pattern}. The \emph{relative scattering} operator
is the map
\begin{equation} \label{eq_rel_sat_op}
  S_\lambda\colon g \mapsto \alpha_s, \qquad S_\lambda\colon L^2(\mathbb S^{n-1}) \to
  L^2(\mathbb S^{n-1}).
\end{equation}
For more technical details see
\cite{colton19_inver_acous_elect_scatt_theor} or more generally
sections 7 and 8 in \cite{agmon76_asymp_proper_solut_differ_equat}.
An $\alpha_s \equiv 0$ implies that $u_s$ decays quickly at infinity.
We see thus define that $\lambda$ is a non-scattering energy if the
kernel of $S_\lambda$ is non-trivial. More precisely
\begin{definition} \label{def_non_scat_energy}
We call $\lambda > 0$ a \emph{non-scattering energy} for a short range
potential $V$ if there is a $g \in L^2(\mathbb S^{n-1})$ such that $g\neq0$
and $S_\lambda g = 0$.
\end{definition}

\noindent
The following Lemma provides another characterization of a
non-scattering-energy.
\begin{lemma} \label{lem_non_scat_alt}
  Let $V$ be a short-range potential. Then $\lambda > 0$ is a
  \emph{non-scattering energy} for $V$ if there exist functions
  $v,w\in B^\ast_2(\mathbb R^n)$ so that $w\not\equiv0$ and
\begin{equation} \label{eq_unbounded_ITE}
  \begin{cases}
    \hfill(-\Delta+V-\lambda)v=0,\\
    \hfill (-\Delta-\lambda)w=0,\\
    \hfill v-w \in \mathring B^\ast_2(\mathbb R^n)
  \end{cases}
\end{equation}
in $\mathbb R^n$.
\end{lemma}

\begin{proof} 
Suppose that $v$ and $w$ are as in the claim. Now define 
\[
u_i := w,\qquad u_s := v-w.
\]
Clearly $u_i$ and $u_s$ solve equations \eqref{eq_scat_parts}. The
incident wave $u_i = w$ can be written as a Herglotz wave function by
Theorem~14.3.3 in
\cite{hoermander05_analy_linear_partial_differ_operat_ii}. In other
words there exists an $g\in L^2(\mathbb S^{n-1})$ so that
\[
u_i(x) = \int_{\mathbb S^{n-1}} e^{i\sqrt{\lambda}\theta \cdot x} g(\theta)
\,d S, \qquad g \in L^2(\mathbb S^{n-1}).
\]
Furthermore since $u_s = v-w \in \mathring B^\ast_2(\mathbb R^n)$, we
see that
\[
\lim_{R \to \infty} \frac{1}{R} \int_{|x|<R} |u_s|^2  \, dx = 0
\]
which would be impossible by \eqref{eq_us_asymp} unless $\alpha_s
\equiv 0$. Hence $S_\lambda g = 0$.
\end{proof}

\medskip
We will also make use of the following Rellich-type theorem which is
Theorem~4 in \cite{vesalainen14_rellic_type_theor_unboun_domain}.
\begin{theorem}\label{rellich-type-theorem}
  Let $n\in\left\{2,3,\ldots\right\}$, $\lambda > 0$, $\gamma > 0$,
  and let $f\in e^{-\gamma\left|\cdot\right|}\,L^2(\mathbb R^n)$ be
  such that $f|_H\equiv0$ for some half-space $H\subset\mathbb
  R^n$. If $u\in\mathring B^\ast_2$ solves the equation
  \[
  \left(-\Delta-\lambda\right)u=f
  \]
  in $\mathbb R^n$ then $u|_H\equiv0$.
\end{theorem}

\medskip
We use the spherical coordinate system given below. Denote by $\mathbb
S^{n-1}$ the unit sphere in $\R^n$. Note that we will be a bit sloppy
with notation, but this will not cause confusion. For example, given a
function $f\colon\mathbb S^2 \to \C$ we will denote by $f(x)$ its value at
a point $x\in\mathbb S^2$. But we might also denote this value by
$f(\vartheta,\varphi)$ and implicitly assume that $x=
(\sin\vartheta\cos\varphi,
\sin\vartheta\sin\varphi,\cos\vartheta)$. For $x\in\R^3$ we use the
spherical coordintes
\begin{equation} \label{eq_spherical_coords}
  \begin{cases}
    x_1=r\,\sin\vartheta\,\cos\varphi,\\
    x_2=r\,\sin\vartheta\,\sin\varphi,\\
    x_3=r\cos\vartheta,
  \end{cases}
\end{equation}
where $r\geq0$, $0\leq\varphi<2\pi$, $0\leq\vartheta\leq\pi$.

Much of our analysis will be based on the use of spherical
harmonics. We denote the space of spherical harmonics of degree $N$ by
$SH^N$. Moreover we will use the basis functions $Y^m_N$, given by
\begin{equation} \label{YNmDef}
  Y_N^m(\varphi,\vartheta) =
  (-1)^m \sqrt{ \frac{2N+1}{4\pi} \cdot \frac{\left(N-m\right)!}{\left(N+m\right)!} }
  \,e^{im\varphi}\,P_N^m(\cos\vartheta)
\end{equation}
for $N\in\N$, $m\in\{-N,-N+1,\ldots,N-1,N\}$ and where $P_N^m$ is the
usual \emph{associated Legendre polynomial}, which are defined by the Legendre polynomials $P_N$ in \eqref{eq_Pmn}. For more
information on these, see
\cite{gallier13_notes_spher_harmon_linear_repres_lie_group,morimoto98_analy,stein71_introd_fourier_euclid}.
We shall have an opportunity to use the following classical result
from the theory of spherical harmonics, known as Laplace's second
integral representation for $P_N^m$. See Section~63 of Chapter~III in
\cite{hobson55}.
\begin{theorem}	\label{thm_Laplace_repr}
  Let $N\in\N$ and let $m$ be an integer such that
  $|m|\leq N$. Then, for $0\leq\vartheta<\pi/2$, we
  have
  \begin{equation}\label{eq_spherical_integral}
    \frac1{2\pi}\int_{0}^{2\pi}\frac{e^{im\psi}\,
      d\psi}{\left(\cos\vartheta+i\sin\vartheta\cos\psi\right)^{N+1}}
    =\frac{\left(N-m\right)!\left(-1\right)^m}{N!}\,P_N^m(\cos\vartheta).
  \end{equation}
\end{theorem}

\begin{remark}
  In fact, slightly more can be said: if $m$ happens to be an integer
  with $\left|m\right|>N$, then the integral on the left-hand side
  vanishes for $0\leq\vartheta<\pi/2$.
\end{remark}

\medskip
Finally, we will be making use of the following lemma, which is
essentially Lemma~3.6 in \cite{blaasten14_corner_alway_scatt} or
Lemma~4.3 in \cite{paeivaerinta17_stric_convex_corner_scatt}.
\begin{lemma}\label{integral-lemma}
  Let $n\in\left\{2,3,\ldots\right\}$, $\beta\geq0$, $1\leq
  q\leq\infty$, $f\in L^q(\mathbb R^n)$, and let $R\colon\R^n\to\C$ be
  a measurable function such that
  \[
  R(x)=O(\left|x\right|^\beta)
  \]
  for $x\in\R^n$. Also, let $\rho\in\C^n \setminus \{0\}$, and let
  $C\subset\R^n$ be a closed cone with vertex at the origin such that
  \[
  \inf_{x\in C\setminus\left\{0\right\}}
  \frac{x\cdot\Re\rho}{\left|x\right|\cdot\left|\Re\rho\right|} > 0.
  \]
  Then
  \[
  \int_Ce^{-\rho\cdot x}\,R(x)\,f(x)\, dx \lesssim
  |\rho|^{n/q-\beta-n} \bigl\| e^{-\Re\rho/|\rho|\cdot x} |x|^\beta
  \chi_C \bigr\|_{L^{q'}(\R^n)} \|f\|_{L^q(\mathbb R^n)},
  \]
  where, as usual, $1\leq q'\leq\infty$ so that $1/q+1/q'=1$.
\end{lemma}
\noindent
In particular, the value $q=\infty$ is acceptable, and then the
exponent $n/q-\beta-n$ reduces to $-\beta-n$ as one would expect. The
lemma is proven by the change of variables $x=y/|\rho|$ and an
application of H\"older's inequality.

\section{The Source Problem}\label{sec_source}

\noindent
In this section we prove Theorem~\ref{thm_source}. We begin with a few definitions.

\begin{definition}\label{admissible-source-corner}
We call a closed cone $C\subset\R^3$ with vertex at the origin an
\emph{admissible source cone} if it is contained in some closed
strictly convex circular cone, if its exterior is connected, and if
its spherical cross section $C\cap\mathbb S^2$ has the property that
\[
\int_{\mathbb S^2\cap C}\,Y_2^m \, dS \neq 0
\]
for some $m\in\left\{-2,-1,0,1,2\right\}$. Here $Y_2^m$ denotes a
spherical harmonic of degree two \eqref{YNmDef}.
\end{definition}

\begin{remark}
We point out that Definition~\ref{admissible-source-corner} is
invariant with respect to rotation. That is, if $C$ satisfies all the
conditions of the definition, then the rotated cone
\[
O[C]=\left\{Ox\,\middle|\,x\in C\right\}
\]
also satisfies the conditions for any $O\in\mathrm{SO}(3,\mathbb R)$.
\end{remark}

\noindent
We will show that smooth enough sources that have a jump on the
boundary near a vertex of an admissible cone produce a non-zero
far-field, see Theorem~\ref{source-theorem}. But before that, we will
give some examples of admissible source cones. In particular, cones
whose vertex angle is small are admissible no matter the shape of
their cross-section.

\begin{definition} \label{def_source_circ}
  Let $\gamma\in (0,\pi/2)$. Then we define $C_\gamma$ to be the
  closed strictly convex circular cone
  \[
  C_\gamma=\left\{x\in\mathbb R^3\middle|\left|x'\right|\leq
  x_3\tan\gamma\right\}
  \]
  where $x'=(x_1,x_2)$. We also define a ``magic angle'' $\vartheta_0$
  to be the unique angle $\vartheta_0\in(0,\pi/2)$ for which
  $\cos\vartheta_0=1/\sqrt3$.
\end{definition}

\begin{remark}
  In degrees, the magic angle $\vartheta_0$ is approximately
  $54.74^\circ$.
\end{remark}

\begin{proposition}
  Let $\gamma\in (0,\pi/2)$. Then $C_\gamma$ is an admissible source
  cone.
\end{proposition}
\begin{proof}
  The cone $C_\gamma$ is contained in a closed strictly convex
  circular cone, namely $C_\gamma$ itself. It is also clear that the
  exterior $\mathbb R^3\setminus C_\gamma$ is connected. To prove the
  integral condition, we shall prove that
  \[
  \int_{C_\gamma\cap\mathbb S^2}Y_2^0(x)\, dx\neq0.
  \]
  After moving to spherical coordinates, writing $Y_2^0$ in terms of
  the Legendre polynomial $P_2$ as in \eqref{YNmDef}, and observing
  that the integrand is independent of the variable $\varphi$, we only
  have to show that
  \[
  \int_0^\gamma\left(3\cos^2\vartheta-1\right)\sin\vartheta\,
  d\vartheta\neq0.
  \]
  But we can compute that
  \[
  \int_0^\gamma\left(3\cos^2\vartheta-1\right)\sin\vartheta\,
  d\vartheta=\cos\gamma\sin^2\gamma,
  \]
  and clearly the product is non-zero.
\end{proof}

\begin{proposition}\label{prop_small_angles}
  If $C\subset\R^3$ is a closed strictly convex cone with vertex at
  the origin, if $C$ is not a set of measure zero, if $C$ has
  connected exterior, and if $C\subseteq C_{\vartheta_0}$, then $C$ is
  an admissible source cone.
\end{proposition}

\begin{proof}
  This follows directly from the observation that, in spherical
  coordinates, the function $Y_2^0$ is strictly positive when
  $\vartheta>\vartheta_0$. Thus, in the integral
  \[
  \int_{C\cap\mathbb S^2}Y_2^0(x)\, dx
  \]
  the integrand is strictly positive for all $x$, except possibly for
  $x$ lying in a set of measure zero, and therefore the integral must
  also be strictly positive.
\end{proof}

\begin{proposition}\label{prop_hollow_source}
  If $C\subset\R^3$ is a closed strictly convex cone with vertex at
  the origin, if $C$ is not a set of measure zero, if $C$ has
  connected exterior, and if $C\subseteq
  C_\gamma\setminus\mathrm{int}\,C_{\vartheta_0}$ for some $\gamma\in
  (\vartheta_0,\pi/2 )$, then $C$ is an admissible source cone.
\end{proposition}
\begin{proof}
  This case can be handled in exactly the same way as the previous
  example, except that now $Y_2^0$ is strictly negative almost
  everywhere in $C\cap \mathbb S^2$.
\end{proof}

\begin{proposition}
  If $C\subset\R^3$ is a closed cone with vertex at the origin, if $C$
  has connected exterior, and if $C_{\gamma_1}\subseteq C\subseteq
  C_{\gamma_2}$ for some $\gamma_1\in (0,\vartheta_0)$ and
  $\gamma_2\in (\vartheta_0,\pi/2 )$ such that
  \[
  \cos\gamma_1 \sin^2\gamma_1 + \cos\gamma_2 \sin^2\gamma_2 >
  \cos\vartheta_0 \sin^2\vartheta_0,
  \]
  then $C$ is an admissible source corner.
\end{proposition}
\begin{proof}
  Again, the point is that, in spherical polar coordinates, $Y_2^0$ is
  strictly positive when $\vartheta<\vartheta_0$ and strictly negative
  when $\vartheta_0<\vartheta<\pi/2$, and so
  \[
  \int_{C\cap\mathbb S^2}Y_2^0 \,dS \geq
  \int_{C_{\gamma_1}\cap\mathbb S^2} Y_2^0\,dS +
  \int_{\left(C_{\gamma_2}\setminus
    C_{\vartheta_0}\right)\cap\mathbb S^2} Y_2^0\,dS.
  \]
  After moving to spherical coordinates, writing $Y_2^0$ in terms of
  $P_2$ and simplifying, we see that the original integral is strictly
  positive if
  \[
  \cos\gamma_1 \sin^2\gamma_1 + \left( \cos\gamma_2 \sin^2\gamma_2 -
  \cos\vartheta_0 \sin^2\vartheta_0 \right) > 0,
  \]
  and this is precisely the hypothesis we made.
\end{proof}

We can prove the main theorem about always scattering sources. Note in
the proof that since the cone $C$ might have a very rough boundary, we
cannot use the simple integration by parts formula from
\cite{blaasten18_nonrad_sourc_trans_eigen_vanis}.
\begin{theorem} \label{source-theorem}
  Let $\gamma > 0$ and let $\Phi\in
  e^{-\gamma\left|\cdot\right|}\,L^2(\R^3)$ be such that
  $\Phi|_H\equiv0$ for some open half-space $H\subset\mathbb R^3$.
  Assume also that the origin belongs to the component of $\R^3
  \setminus \mathrm{supp}\,\Phi$ containing $H$. Let $\varphi\in
  L^\infty_{\mathrm c}(\R^3)$ be such that there exist constants
  $\alpha > 0$ and $\varphi_0\in\mathbb C$ such that
  \[
  \left|\varphi(x)-\varphi_0\right| \lesssim \left|x\right|^\alpha
  \]
  for almost all $x\in\R^3$. Let $C\subset\R^3$ be an admissible
  source cone with vertex at the origin. Finally, assume that
  $u\in\mathring B^\ast_2$ solves the equation
  \[
  \left(-\Delta-\lambda\right)u=\varphi\,\chi_C+\Phi
  \]
  in $\R^3$. Then $\varphi_0=0$.
\end{theorem}
\begin{proof}
  Let $r > 0$ be so small that the ball $B(0,2r)$ is contained in the
  component of $\mathbb R^3\setminus\mathrm{supp}\,\Phi$ containing
  $H$. By Theorem~\ref{rellich-type-theorem}, unique continuation and
  the connectedness of the exterior of $C$ from
  Definition~\ref{admissible-source-corner}, we must have
  $u|_{B(0,2r)\setminus C}\equiv0$. We also observe that, by rotating
  the entire configuration, if necessary, we may assume that $C$ is
  oriented so that $C\setminus\{0\} \subset \R^2\times\R_+$.

  Next, let $\sqrt\lambda \leq \tau < \infty$ and $\psi\in\R$. We
  shall employ the complex geometrical optics solution
  $\exp(-\rho\cdot x)$ where $\rho\in\C^3$ is defined by
  \[
  \rho = \rho(\tau,\psi) = \tau(0,0,1) + i\sqrt{\tau^2+\lambda}
  (\cos\psi,\sin\psi,0).
  \]
  In particular, we have $\rho\cdot\rho=-\lambda$, so that
  \[
  (-\Delta-\lambda) e^{-\rho\cdot x}=0.
  \]
  We study the limit $\tau\to\infty$, and so all the implicit
  constants below will be independent of $\tau$, even though they are
  allowed to depend on $\psi$, $\varphi$, $\alpha$, $r$, $\varphi_0$,
  $C$, $u$, the implicit constant in the theorem statement, and the
  cut-off function $\chi$ chosen below. In the

  We choose a cut-off function $\chi\in C_{\mathrm c}^\infty(\R^3)$ so
  that $\chi|_{B(0,r)}\equiv1$ and that $\chi|_{\R^3\setminus
    B(0,2r)}\equiv0$, and we write $A$ for the annulus
  $B(0,2r)\setminus\overline B(0,r)$. Then, using our knowledge of the
  supports of $\chi, \nabla\chi, \Delta\chi$, $\Phi$ and $u$, we may
  argue that
  \begin{align*}
    0&=\int_{\R^3} \chi\, u (-\Delta-\lambda) e^{-\rho\cdot x}
    \,dx = \int_{\R^3} e^{-\rho\cdot x} (-\Delta-\lambda) (\chi
    u) \,dx \\ &= \int_C e^{-\rho\cdot x}\,\chi\,\varphi\,dx -
    \int_{C\cap A} e^{-\rho\cdot x} (2\,\nabla\chi\cdot\nabla
    u+u\,\Delta\chi)\,dx.
  \end{align*}
  In $C$, the factor $e^{-\rho\cdot x}$ is exponentially decaying, and
  we may estimate, for some constant $\delta > 0$ depending only on
  $C$ and $r$, that $\left|e^{\rho\cdot x}\right|\lesssim
  e^{-\delta\tau}$ for $x\in C\cap A$. We now conclude that
  \[
  \int_{C\cap A} e^{-\rho\cdot x} (2\,\nabla\chi\cdot\nabla
  u+u\,\Delta\chi) \,dx \lesssim e^{-\delta\tau} = o(\tau^{-3}),
  \]
  since the constants are independent of $\tau$.

  We may now continue with
  \[
  o(\tau^{-3}) = \int_{C}e^{-\rho\cdot x}\,\chi\,\varphi \,dx =
  \varphi_0 \int_{C} e^{-\rho\cdot x}\,\chi \,dx +
  \int_{C} e^{-\rho\cdot x}\,\chi (\varphi(x)-\varphi_0) \,dx.
  \]
  We have $|\Re\rho|/|\rho|\geq 1/\sqrt3$. Thus, by
  Lemma~\ref{integral-lemma},
  \[
  \int_{C} e^{-\rho\cdot x}\,\chi (\varphi(x)-\varphi_0) \,dx
  \lesssim |\rho|^{-\alpha-3} \asymp \tau^{-3-\alpha} = o(\tau^{-3}).
  \]
  We may also estimate
  \[
  \int_{C} e^{-\rho\cdot x} (1-\chi)\,dx \lesssim
  e^{-\delta\tau/2} \int_{C\setminus B(0,r)} e^{-\Re\rho\cdot
    x/2} \,dx\lesssim e^{-\delta\tau/2} = o(\tau^{-3}).
  \]
  Thus, we may further continue with a change of variables to get
  \[
  o(\tau^{-3}) = \varphi_0 \int_C e^{-\rho\cdot x}\, dx =
  \frac{\varphi_0}{|\rho|^3} \int_C e^{-\rho/|\rho|\cdot y}\,
  dy,
  \]
  leading to
  \[
  \varphi_0 \int_C e^{-\rho/|\rho|\cdot x} \,dx = o(1).
  \]
  Since $e^{-\rho/|\rho|\cdot x} \lesssim e^{-x_3/\sqrt3}$ in $C$, and
  $\rho/|\rho|\to\rho_0/|\rho_0|$, where
  \[
  \rho_0=\rho_0(\tau,\psi)=\tau(0,0,1)+i\tau(\cos\psi,\sin\psi,0),
  \]
  we get by the dominated convergence theorem, after taking
  $\tau\to\infty$,
  \[
  \varphi_0\int_Ce^{-\rho_0/\left|\rho_0\right|\cdot x} \,dx = 0.
  \]

  Moving next to polar coordinates, we obtain
  \[
  0 = \varphi_0 \int_C e^{-\rho_0/|\rho_0|\cdot x} \,dx =
  \varphi_0 \int_{C\cap\mathbb S^2} \int_0^\infty
  e^{-(\rho_0/|\rho_0|\cdot y)r}\, r^2\, dr\, dy = \varphi_0
  \int_{C\cap\mathbb S^2} \frac2{(\rho_0/|\rho_0|\cdot
    y)^{3}}\, dy.
  \]
  Let $m\in\left\{-2,-1,0,1,2\right\}$. Moving to spherical
  coordinates, multiplying both sides with $e^{im\psi}$, and
  integrating with respect to $\psi$ over the interval $(0,2\pi)$
  gives
  \begin{align*}
    0 &= \varphi_0 \int_0^{2\pi} e^{im\psi}
    \int_{C\cap\mathbb S^2} \frac{\sin\vartheta\, d\varphi\,
      d\vartheta}{(\cos\vartheta + i\sin\vartheta
      \cos(\psi-\varphi))^3}\, d\psi\\ &=\varphi_0
    \int_{C\cap\mathbb S^2} e^{im\varphi} \int_0^{2\pi}
    \frac{e^{im\psi'}\, d\psi'} {(\cos\vartheta + i \sin\vartheta
      \cos\psi')^3}\, \sin\vartheta\, d\varphi\, d\vartheta.
  \end{align*}
  Using Laplace's second representation formula,
  i.e. Theorem~\ref{thm_Laplace_repr}, and multiplying by appropriate
  constant factors gives
  \[
  0 = \varphi_0 \int_{C\cap\mathbb S^2}Y_2^m \,dS.
  \]
  Since, by Definition~\ref{admissible-source-corner}, the last
  integral must be nonzero for some choice of $m$, we conclude that
  $\varphi_0=0$.
\end{proof}

\section{Complex geometrical optics solutions}\label{sec_cgo}

\noindent
Here we construct so called complex geometrical optics solutions
(CGO-solutions for short) for the Schr\"odinger equation with a cone
like potential.  To this end we first show that the admissible cones
give continuous multiplication operators between certain Sobolev
spaces. This argument closely follows the argument in
\cite{paeivaerinta17_stric_convex_corner_scatt}. We then construct the
relevant CGO-solutions, as in that paper, as a consequence of
the $L^p$-resolvent estimates of
\cite{kenig87_unifor_sobol_inequal_unique_contin} and the multiplier
properties of the cone like potential. We end the section by giving
some example cones with fairly complicated structure that admit
CGO-solutions.

We need to show that $V$ is a pointwise multiplier. The idea is to
first establish that if the cross-section of the cone is such that the
characteristic function of the cross-section has suitable properties,
then also a cylinder with the same cross-section has those properties.
This can then be used to cut a infinite cone in to finite pieces with
the relevant property, and using scaling properties of the Sobolev
norms we get the relevant property for the cone itself.

\medskip
Our starting point is a slightly reformulated version of
Proposition~3.9 in \cite{paeivaerinta17_stric_convex_corner_scatt}.
\begin{lemma} \label{lem_straightcyl}
  Let $D \subset \mathbb R^{2}$ be bounded and such that $\chi_D \in
  H^{\tau,2}(\mathbb R^2)$ for some $\tau \in [0,1/2)$, and let $C$ be
    the cylinder
  \[
  C = \{ (x',x_3) \in \mathbb R^3 \;\colon\; x' \in D, \; |x_3| \leq 1
  \}.
  \]
  Then $\chi_C \in H^{\tau,p}(\mathbb R^3)$ for all $p\in[1,2]$.
\end{lemma}

\begin{proof}
  Let us prove the claim for $p=2$. The claim for $p<2$ follows from
  this case because $\chi_C$ has compact support. Set $x:=(x',x_3)$,
  and $I := [-1,1]$. Note that
  \[
  \chi_C(x) = \chi_D(x') \chi_I(x_3),
  \]
  so that
  \begin{align*}
    \widehat \chi_C(\xi) = \widehat \chi_D(\xi') \widehat \chi_I
    (\xi_3).
  \end{align*}
  For $\chi_I$ one has the estimate
  \[
  |\widehat \chi_I(\xi_3) | = \frac{ |\sin(\xi_3) |}{ |\xi_3|} \leq C
  \langle \xi_3 \rangle^{-1}.
  \]
  One has $ \langle \xi \rangle^{2\tau} \leq C \langle \xi'
  \rangle^{2\tau} \langle \xi_3 \rangle^{2\tau}.  $ Using these facts
  gives then that
  \begin{align*}
    \|\chi_C\|^2_{H^{\tau,2}(\mathbb R^3)} &\leq C \int_{\mathbb R^3}
    \langle \xi' \rangle^{2\tau} |\widehat \chi_D (\xi')|^2 \langle
    \xi_3 \rangle^{2\tau} |\widehat \chi_I (\xi_3)|^2\,d\xi \\ &\leq C
    \| \chi_D \|^2_{H^{\tau,2}(\mathbb R^{2})} \int_{\mathbb R}
    \langle \xi_3 \rangle^{2 \tau} \langle \xi_3 \rangle^{-2} \,
    d\xi_3 \\ &< \infty,
  \end{align*}
  provided that $\tau < 1/2$.
\end{proof}

\noindent
The previous lemma can be used as in
\cite{paeivaerinta17_stric_convex_corner_scatt} to deduce the following
proposition.

\begin{lemma}  \label{lem_complete_cone_1}
  Let $\tau \in [0,1/2)$ and suppose $D \subset \mathbb R^{2}$ is
  bounded with $\chi_D \in H^{2,\tau}(\mathbb R^2)$. For $\delta >
  0$ let $C_\delta$ be the cone
  \[
  C_\delta = \{ (t\delta x',t) \in \mathbb R^3 \;\colon\; t \in
  [0,\infty),\; x' \in D \}.
  \]
  Then $\langle x\rangle^{-\alpha}\chi_{C_\delta} \in
  H^{\tau,p}(\mathbb R^3)$ if $p\in(1,2]$ and $\alpha > 3/p$.
\end{lemma}

\begin{proof}
  This lemma is the analogue of Proposition~3.7 in
  \cite{paeivaerinta17_stric_convex_corner_scatt}. Its proof can be
  easily modified to work in our case.  We only need to use our
  Lemma~\ref{lem_straightcyl} instead of Proposition~3.9 in
  \cite{paeivaerinta17_stric_convex_corner_scatt}.
\end{proof}

\noindent
We consider a potential $V$ defined using the characteristic function
of the previous lemma next.

\begin{proposition}  \label{prop_mult}
  Let $C_\delta$ be as in Lemma~\ref{lem_complete_cone_1}.  Consider
  the potential
  \[
  V(x):=\varphi(x)\langle{x}\rangle^{-\alpha}\chi_{C_\delta},
  \]
  where $\tau \in \big(0,\tfrac{1}{2}\big)$, $\alpha >9/4$ and
  $\varphi \in C^{\tau + \varepsilon}(\mathbb R^3)$ for some
  $\varepsilon > 0$. Then
  \begin{equation} \label{eq_mult}
    \| V g \|_{H^{\tau,4/3}(\mathbb R^{3})} \leq C
    \|g\|_{H^{\tau,4}(\mathbb R^{3})}
  \end{equation}
  for any $g\in H^{\tau,4}(\mathbb R^3)$ and
  \begin{equation} \label{eq_itself}
    \| V \|_{H^{\tau,4/3}(\mathbb R^3)} < \infty.
  \end{equation}
\end{proposition}

\begin{proof}
  Use the notation
  \[
  \widetilde V := \langle{x}\rangle^{-\alpha}\chi_{C_\delta},
  \]
  so that $V=\varphi \widetilde V$. Since $\widetilde V \in
  H^{\tau,2}(\mathbb R^3)$ by Lemma~\ref{lem_complete_cone_1}, we get
  as a direct consequence of Proposition~3.5 in
  \cite{paeivaerinta17_stric_convex_corner_scatt} that the multiplication operator
  $g \mapsto \widetilde V g$ is continuous $H^{\tau, 4}(\mathbb R^3)
  \to H^{\tau, 4/3}(\mathbb R^3)$, and hence the estimate
  \[
  \| \widetilde V g \|_{H^{\tau, 4/3}(\mathbb R^3)} \leq C \| g
  \|_{H^{\tau, 4}(\mathbb R^3)}
  \]
  follows. In more detail, choose $\tilde p = 2$ in Proposition 3.5 in
  \cite{paeivaerinta17_stric_convex_corner_scatt}, which gives $p = 4$ and the
  H\"older conjugate $p' = \tfrac{4}{3}$.  Furthermore by the results
  on p. 205 in \cite{triebel92_theor_funct_spaces_ii}, we see that multiplication by $\varphi \in
  C^{\tau + \varepsilon}(\mathbb R^3)$ is continuous on
  $H^{\tau,p}(\mathbb R^3)$, so that
  \[
  \| V g \|_{H^{\tau, 4/3}(\mathbb R^3)} = \| \varphi \widetilde V g
  \|_{H^{\tau, 4/3}(\mathbb R^3)} \leq C \| \widetilde V g
  \|_{H^{\tau, 4/3}(\mathbb R^3)} \leq C \| g \|_{H^{\tau, 4}(\mathbb
    R^3)}.
  \]
  Lemma~\ref{lem_complete_cone_1} with $p=4/3$ and this same result in
  \cite{triebel92_theor_funct_spaces_ii} shows \eqref{eq_itself} immediately.
\end{proof}

The previous proposition shows that potentials of the form $\varphi
\langle x \rangle^{-\alpha} \chi_{C_\delta}$ give continuous
multiplier operators in the sense of \eqref{eq_mult}. The uniform
Sobolev estimates of \cite{kenig87_unifor_sobol_inequal_unique_contin} and multiplier property
\eqref{eq_mult} can be used to deduce the following existence result
for a inhomogeneous equation. Then, the norm bound \eqref{eq_itself}
implies the existence of complex geometrical optics solutions for the
homogeneous partial differential equation $(-\Delta+V-\lambda)\psi =
0$. This is a standard argument, done for potentials supported on a
circular cone in Proposition~3.3 of
\cite{paeivaerinta17_stric_convex_corner_scatt}. We nevertheless repeat parts of
the proof here for the convenience of the reader.

\begin{proposition}  \label{prop_solv}
  Let $\tau \in \mathbb R$. And suppose $V\in \mathcal{D}'(\mathbb
  R^3)$ is such that
  \[
  \| V g \|_{H^{\tau, 4/3}(\mathbb R^3)} \leq C \| g \|_{H^{\tau,
      4}(\mathbb R^3)}
  \]
  for all $g\in H^{\tau,4}(\mathbb R^3)$. Then there exist constants
  $C > 0$ and $R > 0$ such that whenever $\zeta \in \mathbb C^n$
  satisfies $|\Re\zeta| \geq R$ and $f \in H^{\tau,4/3}(\mathbb R^3)$
  the equation
  \begin{align}\label{eq_solv}
    (-\Delta + 2 \zeta \cdot D + V) \psi = f
  \end{align}
  has a solution $\psi \in H^{\tau,4}(\mathbb R^3)$ satisfying
  \[
  \|\psi\|_{H^{\tau,4}(\mathbb R^3)} \leq C |\Re
  \zeta|^{-1/2}\|f\|_{H^{\tau,4/3}(\mathbb R^3)}.
  \]
  
\end{proposition}

\begin{proof} 
  Here we only repeat the case involving a nonzero $V$.  The first
  part of the proof of Proposition~3.3 in
  \cite{paeivaerinta17_stric_convex_corner_scatt}, gives an solution operator
  $G_\zeta$ for the case $V=0$, for which
  \[
  G_\zeta \colon H^{\tau,4/3}(\mathbb R^3) \to H^{\tau,4}(\mathbb
  R^3),
  \]
  and $(-\Delta + 2 \zeta \cdot D )G_\zeta g = g$ for any $g\in
  H^{\tau,4/3}(\mathbb R^3)$. Furthermore $G_\zeta$ obeys the norm
  estimate
  \[
  \|G_\zeta g \|_{H^{\tau,4}(\mathbb R^3)} \leq C |\Re \zeta|^{-1/2}\| g
  \|_{H^{\tau,4/3}(\mathbb R^3)}
  \]
  for all $g$ in that space. Proposition~\ref{prop_mult} implies that
  \[
  \|VG_\zeta g \|_{H^{\tau,4/3}(\mathbb R^3)} \leq C |\Re \zeta|^{-1/2}\|
  g \|_{H^{\tau,4/3}(\mathbb R^3)}
  \]
  again for all $g$ as before. We can now choose an $R>0$, so that
  $CR^{-1/2} = \tfrac{1}{2}$, and then assuming $|\Re(\zeta)| \geq R$
  we have
  \begin{align} \label{eq_potest}
    \|VG_\zeta g \|_{H^{\tau,4/3}(\mathbb R^3)} \leq \tfrac{1}{2} \| g
    \|_{H^{\tau,4/3}(\mathbb R^3)}
  \end{align}
  for any $g$. Note that a solution to the equation $(-\Delta + 2
  \zeta \cdot D + V) G_\zeta v = f$ is given by any solution to $v + V
  G_\zeta v = f$. If we can solve $v$ satisfying
  \[
  (I + V G_\zeta)v = f,
  \]
  then we get a solution $\psi$ to \eqref{eq_solv} by setting $\psi =
  G_\zeta v$. Estimate \eqref{eq_potest} shows that $(I +
  VG_\zeta)^{-1}$ can be expressed as a Neumann series converging in
  $H^{\tau,4/3}(\mathbb R^3)$. By using the norm estimate for
  $G_\zeta$ and \eqref{eq_potest} we have moreover that
  \[
  \| \psi \|_{H^{\tau,4}(\mathbb R^3)} \leq C |\Re \zeta|^{-1/2}\| v
  \|_{H^{\tau,4/3}(\mathbb R^3)} \leq C |\Re \zeta|^{-1/2}\| f
  \|_{H^{\tau,4/3}(\mathbb R^3)}.
  \]
\end{proof}

\noindent The solvability result gives us then finally the existence
of suitable CGO solutions.

\begin{theorem}\label{cgo-solutions}
  Let $\lambda\in\mathbb R_+$, let $\rho\in\mathbb C^3$ satisfy
  $\rho\cdot\rho=-\lambda$, and assume that $\left|\Im\rho\right|$ is
  large enough. Define
  \[V(x)=\varphi(x)\left\langle x\right\rangle^{-\alpha}\chi_C(x)\]
  for $x\in\mathbb R^3$, where $C\subset\mathbb R^3$ is a closed cone
  with a bounded cross-section $D\subset\mathbb R^2$ satisfying
  $\chi_D\in H^\tau(\mathbb R^2)$,
  $\alpha\in (9/4,\infty )$, and $\varphi\in
  C^{\tau+\varepsilon}(\mathbb R^3)$ for some
  $\tau\in (1/4,1/2 )$ and $\varepsilon\in\mathbb R_+$.

  Then there exists a solution $u=e^{-\rho\cdot x}\left(1+\psi\right)$
  to the equation
  \[
  \left(-\Delta+V-\lambda\right)u=0
  \]
  in $\mathbb R^3$, where $\psi\in L^q(\mathbb R^3)$, $q=12/(3-4\tau)
  \in (6,12)$, and
  \[
  \bigl\|\psi\bigr\|_{L^q(\mathbb R^3)} \lesssim
  \left|\Im\rho\right|^{-3/q-\delta}
  \]
  for some fixed $\delta\in\mathbb R_+$.
\end{theorem}

\begin{proof}
  To simplify notation, let $\zeta = i\rho$. It is then required to
  construct a solution of the form
  \[
  u = e^{i \zeta \cdot x} (1 + \psi), \quad \zeta \in \mathbb C^3,
  \quad \zeta\cdot \zeta = \lambda.
  \]
  We thus need to solve the equation
  \[
  (-\Delta + 2 \zeta \cdot D + V) \psi = -V.
  \]
  The potential $V$ satisfies the two conditions
  \[
  \| V g \|_{H^{\tau, 4/3}(\mathbb R^3)} \leq C \| g \|_{H^{\tau,
      4}(\mathbb R^3)}, \qquad \| V \|_{H^{\tau,4/3}(\mathbb R^3)} <
  \infty
  \]
  by Proposition~\ref{prop_mult}. The former is true for any $g\in
  H^{\tau,4}(\mathbb R^3)$ and $C$ is independent of $g$.
  Proposition~\ref{prop_solv} gives us then a solution $\psi\in
  H^{\tau,4}(\mathbb R^3)$ satisfying the estimate
  \begin{align} \label{eq_krs}
    \|\psi\|_{H^{\tau,4}(\mathbb R^3)} \leq C |\Re \zeta|^{-1/2}\| V
    \|_{H^{\tau,4/3}(\mathbb R^3)}.
  \end{align}
  By the Sobolev embedding Theorem we have that
  \[
  \|\psi\|_{L^{q_\tau}(\mathbb R^3)} \leq C
  \|\psi\|_{H^{\tau,4}(\mathbb R^3)},
  \]
  when $0 < \tau < n/p$ and $q_\tau = np/(n-\tau p)$ where $n=3$ is
  the dimension and $p=4$ is the Sobolev integrability. We have
  \[
  q_\tau = \tfrac{12}{3-4\tau} \in (6, 12).
  \]
  since $\tau \in (1/4,1/2)$ by assumption. Now we write
  \[
  \tfrac{1}{2} = \tfrac{3}{q_\tau} + \tfrac{q_\tau-6}{2q_\tau} =:
  \tfrac{3}{q_\tau} + \delta
  \]
  and for the estimate of $\psi$ in the claim we want to have $\delta
  > 0$, whichis equivalent to having $q_\tau > 6$. This holds because
  $\tau > 1/4$. Combining the above with \eqref{eq_krs} implies that
  \[
  \|\psi\|_{L^{q_\tau}(\mathbb R^3)} \leq C |\Re \zeta|^{(-3/q_\tau -
    \delta)}\| V \|_{H^{\tau,4/3}(\mathbb R^3)},
  \]
  where $q_\tau \in (6,12)$.
\end{proof}

We end this section by considering a concrete examples of the type of
potential $V$ for which the above results hold.

\begin{example}
  \begin{figure}
    \includegraphics{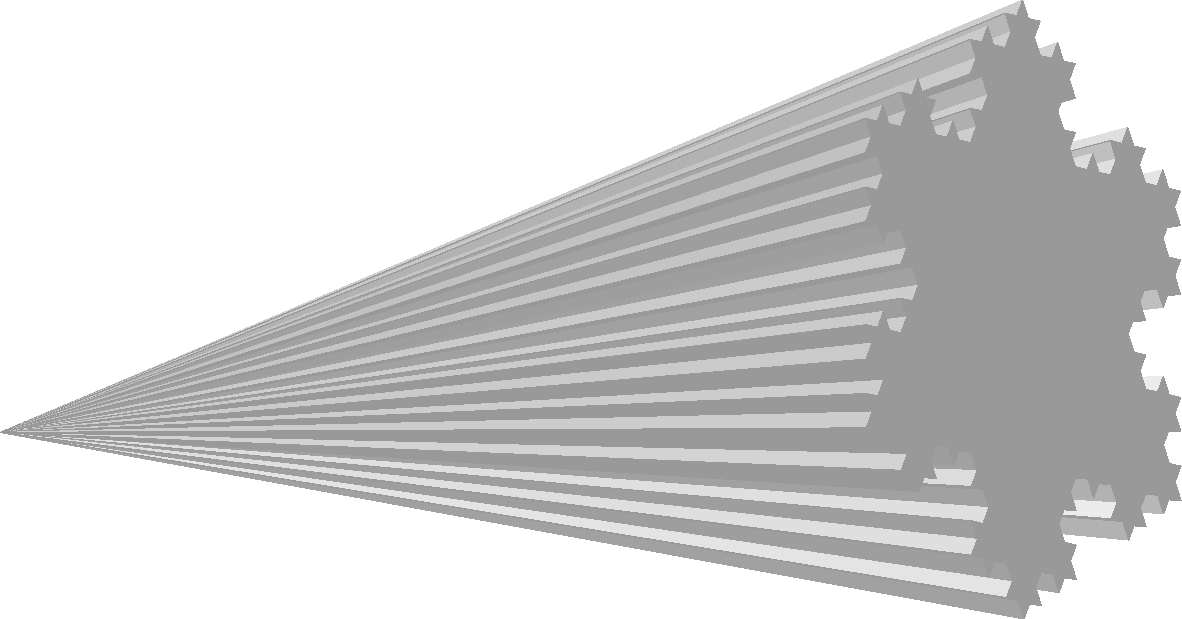}
    \caption{A cone with Koch snowflake cross section}
    \label{figure_Koch}
  \end{figure}
  Consider the following conic set
  \[
  C = \{ (tx',t) \in \mathbb R^3 \;\colon\; x' \in K, \; t \in [0,
  \infty) \},
  \]
  where $K \subset \mathbb R^2$ is the interior of the famous Koch
  snowflake, see Figure~\ref{figure_Koch}. By Corollary 1.2 in
  \cite{faraco12_sobol_norm_charac_funct_with} we know that
  \[
  \tau < 1 - \tfrac{1}{2}\tfrac{\log 4}{\log 3} \quad \Rightarrow
  \quad \chi_K \in H^{\tau,2}(\mathbb R^2).
  \]
  We can now choose $\tau \in (\tfrac{1}{4},\tfrac{1}{2})$, such that
  $\tfrac{1}{4}< \tau < 1 - \tfrac{1}{2}\tfrac{\log 4}{\log 3} \approx
  0.37$. Let $\alpha>9/4$ and
  \[
  V(x) = \langle x \rangle^{-\alpha} \chi_C.
  \]
  Theorem~\ref{cgo-solutions} implies then that we can construct CGO
  solutions $V$ with the desired remainder estimates.
\end{example}

\medskip
The next example is a cone with a porous structure. We will use some
further results from \cite{faraco12_sobol_norm_charac_funct_with} to analyze this. We use in
particular the fact that we can determine what Sobolev space the
characteristic function of a set is in by looking at the box-counting
dimension of the boundary of the set.

\begin{example}
  Let $A$ be the Apollonian gasket. See Figure~\ref{figure_apollo}.
  The set $A$ can be constructed iteratively by starting with the four
  largest circles. One then adds all the circles that are mutually
  tangent to any three of the initial four circles. In this way, one
  obtains four new circles, or eight circles in total. One then adds
  the circles that are mutually tangent to the eight circles
  similarly, and so forth.

  \begin{figure}
    \begin{center}
      \includegraphics[width=10cm]{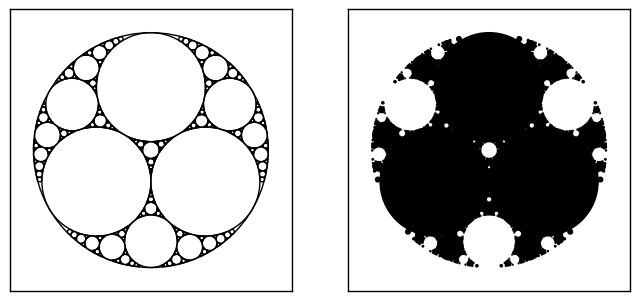}
    \end{center}
    \caption{The Apollonian gasket $A$ on the left and the set $\tilde
      A$ on the right.}
    \label{figure_apollo}
  \end{figure}

  Now consider the set $\tilde A \subset \mathbb R^2$, which is
  constructed as the Apolloninan gasket, but so that $\tilde A$
  contains the entire interior disk of the circles, but only of those
  added on every other iteration in the construction of $A$. See
  Figure~\ref{figure_apollo}.

  Now we consider the conic set
  \[
  C = \{ (tx',t) \in \mathbb R^3 \;\colon\; x' \in \tilde A, \; t \in
  [0, \infty) \},
  \]
  The set $\tilde A$ is a porous set with non zero Lebesgue measure,
  which can be analyzed in a straightforward manner.  Clearly we have
  that
  \[
  \partial \tilde A \subset \partial A = A.
  \]
  The exact Hausdorff dimension $\dim_H(A)$ of the Apolloninan Gasket
  is unknown, but $ \dim_H(A) < 1.314534, $ see
  e.g. \cite{boyd73_resid_set_dimen_apoll_packin}. The Apollonian gasket has furthermore the
  property that the so called box-counting dimension (or Minkowsksi
  dimension) of its closure coincides with the Hausdorff dimension,
  i.e. $\dim_M(\overline{A}) = \dim_H(\overline{A})$, see
  \cite{stratmann96_box_count_dimen_geomet_finit_klein_group}.  Thus we have that
  \[
  \dim_M (\overline{A}) = \dim_H(\overline{A}) < 1.314534.
  \]
  By the monotonicity of the box-counting dimension (see p.48 in
  \cite{falconer03_fract}) we obtain that $\dim_M (\partial \tilde A) \leq
  \dim_M(\overline{A}) \leq 1.32$.  We can now apply the results in
  \cite{faraco12_sobol_norm_charac_funct_with} (see also \cite{sickel99_point_multip_lizor_trieb_spaces}).  Let $D \subset
  \mathbb R^n$, then in general we have that
  \begin{equation} \label{eq_dimcond}
    \dim_M(\partial D) < n - p\tau \quad \Rightarrow \quad \chi_D \in
    H^{\tau,p}(\mathbb R^n),
  \end{equation}
  for $1\leq p<\infty$. See Theorem 1.3 in \cite{faraco12_sobol_norm_charac_funct_with} and
  the related comments. For $\partial \tilde A$ and $p=n=2$, the
  condition becomes
  \[
  \dim_M(\partial\tilde A) < 2 - 2\tau,
  \]
  which follows if $\tau < 1 - 0.5 \cdot 1.32 = 0.34$. Again, we see
  that Theorem~\ref{cgo-solutions} implies that we can construct CGO
  solutions for the potential $V = \langle x \rangle^{-\alpha} \chi_C$
  with the desirable remainder estimates if $\alpha>9/4$, since we can
  choose $\tau$ so that $1/4 < \tau < 0.34$.
\end{example}

\begin{remark}
  Condition \eqref{eq_dimcond} has a partial converse that states that
  for sets $D \subset \mathbb R^n$, we have that
  \begin{align*}
    \dim_M(\partial D) > n - p\tau \quad \Rightarrow \quad \chi_D
    \notin H^{\tau,p}(\mathbb R^n),
  \end{align*}
  for $1\leq p<\infty$. See \cite{faraco12_sobol_norm_charac_funct_with}. This sets some
  limits on the applicability of the above constructions. It follows
  that we cannot obtain CGO-solutions with the above method for
  potentials supported on cones whose cross-section $D \subset \mathbb
  R^2$ has a fractal boundary with $\dim_M(\partial D)$ large enough,
  since the CGO construction demands that $\tau>1/4$.
\end{remark}

\section{Admissible cones always scatter}\label{sec_admissible}

\noindent
In this section we define the notion of an admissible medium cone. Admissibility amounts
essentially to an orthogonality condition and certain regularity assumptions.
We then 
prove that admissible medium cones result in potentials that always scatter. 
After this we formulate a simple determinant
condition for medium cones that are admissible.

\begin{definition}\label{def_admissible-medium-corner}
We call a closed cone $C\subset\mathbb R^3$ an \emph{admissible medium cone}, if
\begin{enumerate}[(i)]
\item $C$ is contained in a strictly convex closed circular cone,
\item $C$ has a connected exterior,
\item $C$ has a bounded cross-section $D\subset\mathbb R^2$ such that
  $\chi_D\in H^\tau(\mathbb R^2)$ for some $\tau\in (1/4,1/2 )$, and
\item for any spherical harmonic $H\not\equiv0$ of arbitrary degree
  $N$, there exists an index $m\in\left\{-N-2,\ldots,N+2\right\}$ so
  that
  \[
  \int_{C\cap \mathbb S^2}Y_{N+2}^m\,H \,dS \neq 0.
  \]
\end{enumerate}
\end{definition}

\begin{remark}
We could equivalently formulate the last condition this way: For any
spherical harmonic $H\not\equiv0$ of arbitrary degree $N$, there
exists a spherical harmonic $Y$ of degree $N+2$ so that
\[
\int_{C\cap \mathbb S^2}Y\,H \,dS  \neq 0.
\]
\end{remark}

\begin{theorem} \label{thm_admScatter}
Let $\lambda\in\mathbb R_+$, and let $V=\varphi\,\chi_C+\Phi$, where
$\varphi\in C^{1/4+\varepsilon}_{  c}(\mathbb R^3)$ for some
$\varepsilon\in\mathbb R_+$, $\Phi\in
e^{-\gamma\left|\cdot\right|}\,L^2(\mathbb R^3)$ for some $\gamma\in\mathbb
R_+$ so that $\Phi|_H\equiv0$ for some open half-space $H\subset\mathbb R^3$,
and so that the origin belongs to the component of $\mathbb
R^3\setminus\mathrm{supp}\,\Phi$ containing $H$, and where $C$ is an admissible
medium cone with vertex at the origin. Finally, assume that $\lambda$ is a
non-scattering energy for the potential $V$. Then $\varphi(0)=0$.
\end{theorem}

\begin{proof}
By Lemma~\ref{lem_non_scat_alt}, there exist solutions $v,w\in B_2^\ast(\mathbb R^3)$ to the equations
\[
\begin{cases}
\left(-\Delta+V-\lambda\right)v=0,\\
\left(-\Delta-\lambda\right)w=0,\end{cases}
\]
in $\mathbb R^3$, so that $u=v-w\in\mathring B^\ast_2(\mathbb R^3)$ and
$w\not\equiv0$. Let $r\in\mathbb R_+$ be so small that the ball $B(0,2r)$ is
contained in the component of $\mathbb R^3\setminus\mathrm{supp}\,\Phi$
containing $H$. By Theorem~\ref{rellich-type-theorem}, unique continuation and
the connectedness of the exterior of $C$ from Definition~\ref{def_admissible-medium-corner}, we have $u|_{B(0,2r)\setminus C}\equiv0$.

By rotating the whole setup, we may assume that $C\setminus\{0\} \subset
\mathbb R^2 \times \mathbb R_+$. Next, we let
$\tau\in\bigl[\sqrt\lambda,\infty\bigr[$ and $\psi\in\mathbb R$ in order to
choose a complex vector $\rho\in\mathbb C^3$ through
\[
\rho=\rho(\tau,\psi)=\tau (0,0,1)
+i\sqrt{\tau^2+\lambda} \,(\cos\psi, \sin\psi, 0).
\]
In particular, we have $\rho\cdot\rho=-\lambda$, and, assuming that $\tau$ is
large enough, Theorem~\ref{cgo-solutions} gives us a solution
$u_0=e^{-\rho\cdot x}\left(1+\psi\right)$ to the equation
\[
\left(-\Delta+\varphi\,\chi_C-\lambda\right)u_0=0
\]
in $\mathbb R^3$ satisfying $\psi\in L^q(\mathbb R^3)$ for some $q\in (6,12)$, and satisfying the estimate
\[
\|\psi\|_{L^q(\mathbb R^3)} \lesssim \tau^{-3/q-\delta},
\]
where $\delta\in\mathbb R_+$ is fixed. We shall study the limit
$\tau \to \infty$, and so all the implicit constants will be
independent of $\tau$, but they are allowed to depend on everything else. 

Since $w$ is real-analytic and $w\not\equiv0$, there exists a homogeneous
complex polynomial $H(x)$ of some degree $N\in\left\{0\right\}\cup\mathbb Z_+$
such that $H(x)\not\equiv0$ and
\[
w(x)=H(x)+O(\left|x\right|^{N+1}),
\]
for all $x\in B(0,2r)$. The equation $\left(-\Delta-\lambda\right)w=0$ implies
that $H$ is harmonic. This can be seen by applying the differential operator to
the Taylor expansion of $w$. The argument is essentially the same as in
Lemma~17 of \cite{blaasten20_non_scatt_energ_trans_eigen_h}. 

We choose a cut-off function $\chi\in C_{\mathrm c}^\infty(\mathbb R^3)$ so
that $\chi|_{B(0,r)}\equiv1$ and $\chi|_{\mathbb R^3\setminus B(0,2r)}\equiv0$,
and write $A$ for the annular domain $B(0,2r)\setminus\overline B(0,r)$. We may
then argue, remembering that $V$ and $\varphi\,\chi_C$ coincide in the support
of $\chi$, that
\begin{align*}
0&=-\int_{\mathbb R^3}\chi u\left(-\Delta+\varphi\,\chi_C-\lambda\right)u_0 \,dx
=-\int_{\mathbb R^3}u_0\left(-\Delta+V-\lambda\right)\left(\chi u\right)\,dx\\
&=\int_Cu_0\,\chi\varphi w\,dx
+\int_{C\cap A}u_0\left(2\,\nabla\chi\cdot\nabla u+u\,\Delta\chi\right)\,dx.
\end{align*}
Since we may estimate $\left|e^{-\rho\cdot x}\right|\lesssim e^{-\delta'\tau}$ in
the domain $A$ for some constant $\delta'\in\mathbb R_+$ not depending on
$\tau$, we have
\[
\int_{C\cap A}u_0\left(2\,\nabla\chi\cdot\nabla u+u\,\Delta\chi\right)\,dx
\lesssim e^{-\delta'\tau}=o(\tau^{-N-3}).
\]
Now, we may expand the integrand $u_0\,\chi\varphi w$ in the remaining integral in steps to get
\begin{align*}
o(\tau^{-N-3})
=\int_Cu_0\,\chi\varphi w \,dx
&=\int_Ce^{-\rho\cdot x}\,\psi\chi\varphi w \,dx
+\int_Ce^{-\rho\cdot x}\,\chi\,\bigl(\varphi-\varphi(0)\bigr)\,w \,dx\\
&\qquad
+\varphi(0)\int_Ce^{-\rho\cdot x}\,\chi\,O(\left|x\right|^{N+1}) \,dx
+\varphi(0)\int_Ce^{-\rho\cdot x}\,\chi H \,dx.
\end{align*}
Using Lemma~\ref{integral-lemma} and the remainder estimate of Theorem~\ref{cgo-solutions} 
for $\psi$, we have for the first of the four integrals on the right-hand side that 
\begin{align*} 
\int_Ce^{-\rho\cdot x}\,\psi\chi\varphi w \,dx  	
\lesssim 
|\rho|^{3/q-N-3}
\big \| e^{-\Re\rho/\left|\rho\right|\cdot x}\left|x\right|^{N}\chi_C \big\|_{L^{q'}(\mathbb R^n)}
\|\chi \varphi \psi \|_{L^q(\mathbb R^n)}
\lesssim 
\tau^{-N-3-\delta}.
\end{align*}
For the second integral we use that H\"older continuity which implies that 
$|\varphi(x) - \varphi(0)| \leq C |x|^{1/4+\epsilon}$ and Lemma~\ref{integral-lemma} with $q=\infty$, so that
have that 
\begin{align*} 
\int_Ce^{-\rho\cdot x}\,\chi\,\bigl(\varphi-\varphi(0)\bigr)\,w \,dx
\lesssim
|\rho|^{-N-3}
\big \| e^{-\Re\rho/\left|\rho\right|\cdot x}\left|x\right|^{N}\chi_C \big\|_{L^{1}(\R^n)}
\|\chi \varphi \psi \|_{L^\infty(\R^n)}
\lesssim 
\tau^{-N-3-1/4-\epsilon}.
\end{align*}
For the third integral we use Lemma~\ref{integral-lemma} similarly and obtain that
\begin{align*} 
\varphi(0)\int_Ce^{-\rho\cdot x}\,\chi\,O(\left|x\right|^{N+1}) \,dx
\lesssim
\tau^{-N-1-3}.
\end{align*}
In particular, all three are $o(\tau^{-N-3})$, and we may continue
\[
o(\tau^{-N-3})=\varphi(0)\int_Ce^{-\rho\cdot x}\,\chi H \,dx
=\varphi(0)\int_Ce^{-\rho\cdot x}\,H \,dx
-\varphi(0)\int_Ce^{-\rho\cdot x}\left(1-\chi\right)H \,dx.
\]
The last integral may be estimated to be
\[
\varphi(0)\int_Ce^{-\rho\cdot x}\left(1-\chi\right)H \,dx
\lesssim e^{-\delta'\tau/2}\int_{C\setminus B(0,r)}e^{-\Re\rho\cdot x/2}\,|H|\,dx
\lesssim e^{-\delta'\tau/2}
=o(\tau^{-N-3}).
\]
Now with a change of variables,
\[
o(\tau^{-N-3})=\varphi(0)\int_Ce^{-\rho\cdot x}\,H\,dx
=\frac{\varphi(0)\,2^{(N+3)/2}}{\left|\rho\right|^{N+3}}\int_Ce^{-\sqrt2\rho/\left|\rho\right|\cdot x}\,H\,dx,
\]
or more simply,
\[
\varphi(0)\int_Ce^{-\sqrt2\rho/\left|\rho\right|\cdot x}\,H \,dx = o(1).
\]
In the limit $\tau \to \infty$, we have that
$\rho/\left|\rho\right| \to \rho_0/\left|\rho_0\right|$, where
$$
\rho_0=\rho_0(\tau,\psi)=\tau (0,0,1) +i\tau \,(\cos\psi, \sin\psi, 0).
$$
And Since $e^{-\rho/\left|\rho\right|\cdot x}\lesssim e^{-x_3/\sqrt3}$ in $C$,
the dominated convergence Theorem gives, when taking $\tau \to \infty$, that
\[
\varphi(0)\int_Ce^{-\sqrt2\rho_0/\left|\rho_0\right|\cdot x}\,H \,dx=0.
\]
Writing $\omega$ for the vector $\sqrt2\rho_0/\left|\rho_0\right|=\left\langle
i\cos\psi,i\sin\psi,1\right\rangle$, and remembering that $\omega\cdot y$ has a
positive real part for $y\in C\cap \mathbb S^2$, we may compute through the use of
polar coordinates that
\[
0=\varphi(0)\int_Ce^{-\omega\cdot x}\,H \,dx
=\varphi(0)\int_{C\cap \mathbb S^2}\int_0^\infty e^{-\left(\omega\cdot y\right)r}\,r^{N+2}\, dr\,H(y)\, dy
=\frac{\varphi(0)}{\left(N+2\right)!}\int_{C\cap \mathbb S^2}\frac{H(y)\, dy}{\left(\omega\cdot y\right)^{N+3}}.
\]
Let $m\in\left\{-N-2,-N-1,\ldots,N+2\right\}$. Moving to spherical coordinates,
multiplying the last identity by $e^{im\psi}$, and integrating with respect to
$\psi$ over $\left[0,2\pi\right]$, changing the order of integration and making
the change of variables $\psi'=\psi-\varphi$ gives
\[
0=\varphi(0)\int_{C\cap
\mathbb S^2}e^{im\varphi}\int_0^{2\pi}\frac{e^{im\psi'}\,
d\psi'}{\left(\cos\vartheta+i\sin\vartheta\cos\psi'\right)^{N+3}}\,H\sin\vartheta\,
d\vartheta\, d\varphi,
\]
from which Laplace's second representation theorem, i.e. Theorem~\ref{thm_Laplace_repr}, gives us
\[
0=\varphi(0)\int_{C\cap \mathbb S^2}
e^{im\varphi}\,P_{N+2}(\cos\vartheta)\,H\sin\vartheta\, d\vartheta\, d\varphi
=\varphi(0)\int_{C\cap \mathbb S^2}Y_{N+2}^m\,H \,dS.
\]
Since $C$ was assumed to be an admissible medium cone, the last integral is
non-zero for some choice of $m$, and thus $\varphi(0)=0$.
\end{proof}

\noindent
In the rest of this section we show that the admissibility of a medium cone can be ascertained by  
studying  certain determinants, which we will shortly define.

\medskip
\noindent
We are particularly interested in 
families of cones $\{C_\rho\;:\;\rho \in I\subset \R\}$ depending on the parameter $\rho$.
Consider the $(2N+1)$-tuple $\Psi$ of basis functions of $SH^{N+2}$, such that
$$
\Psi = (\psi_{-N},..,\psi_{N}), \quad \psi_i\neq \psi_j,\quad
\psi_j \in \big \{Y^{k}_{N+2}\;:\; k = -(N+2),..,N+2 \big\}.
$$
Define the functions
$$
I_N^{k,l}(\rho) : = \big(\psi_k, Y_{N}^{l} \big)_{L^2(K_\rho)} =\int_{K_\rho} \psi_k \ov{Y}_{N}^{l} \,dS,
$$
where $K_\rho := C_\rho \cap \mathbb S^2$.

These give us information of the projections of the basis functions $Y^l_N$ on the  basis functions in $\Psi_N$
in the  $L^2(K_\rho)$-space.
Define the matrices
$$
\mathcal{C}_{\Psi,N}(\rho) :=
\begin{pmatrix}
I_N^{-N,-N}(\rho) & \dots & I_N^{-N,N}(\rho)  \\
\vdots & \ddots & \vdots \\
I_N^{N,-N}(\rho) & \dots & I_N^{N,N}(\rho) 
\end{pmatrix} . 
$$
And let the corresponding determinants be
$$
\mathcal{D}_{\Psi,N}(\rho) := \det \mathcal{C}_{\Psi,N}(\rho). 
$$
Note that we are mainly interested in a specific choice of $\Psi$, which is
$$
\Psi_0 := \big( Y^{-N}_{N+2},Y^{-N+1}_{N+2},..,Y^{N}_{N+2} \big),
$$
as this choice of $\Psi$ is particularly useful for analyzing circular cones. 
Moreover we will use the abbreviations
\begin{equation} \label{eq_defCNDN}
\mathcal{C}_{N}(\rho) := \mathcal{C}_{\Psi_0,N}(\rho), \qquad
\mathcal{D}_{N}(\rho) :=  \mathcal{D}_{\Psi_0,N}(\rho).
\end{equation}
The columns of the matrices $\mathcal{C}_{\Psi,N}$ gives the coefficients of the projections
of the basis vectors $Y^k$ to vectors in $\Psi$.
The next lemma gives a simple condition for admissibility in terms of the 
determinants $\mathcal{D}_{\Psi,N}(\rho)$.

\begin{lemma} \label{lem_DetCond}
Suppose that $C_\rho$ satisfies the regularity conditions (i)-(iii)
of Defintion~\ref{def_admissible-medium-corner}. Then
$C_\rho$ is an admissible medium cone if for every $N \in \N$,
there is a $\Psi_N$,\, for which $\mathcal{D}_{\Psi_N,N}(\rho) \neq 0$.
\end{lemma}

\begin{proof}
Assume that $C_{\rho}$ is not 
an admissible medium cone, so that there exists an $N$ and  $H \in SH^N$, $H\neq0$, s.t.
$$
\int_{K_\rho} Y H \,dS= 0,
$$
for all $Y \in SH^{N+2}$.  We can write $H$, as
$$
H = \sum_{j=-N}^N a_l Y^l_N. 
$$
Then in particular
\begin{align*} 
\int_{K_{\rho}} (a_{-N} Y_N^{-N} + &\dots + a_N Y_N^{N}) \ov{Y}_{N+2}^{-N-2} \,dS = 0, \\
&\vdots &  \\
\int_{K_{\rho}} (a_{-N} Y_N^{-N} + &\dots + a_N Y_N^{N}) \ov{Y}_{N+2}^{N+2} \,dS = 0. 
\end{align*}
So that for any choice of $\Psi$, the vector $\bar \alpha := (\alpha_{-N},..,\alpha_N) \neq 0$ is such that
$$
\mathcal{C}_{\Psi,N}(\rho) \bar\alpha = 0.
$$
This implies that all the matrix 
$\mathcal{C}_{\Psi,N}(\rho)$ is singular, for any choice of $\Psi$.
In particular we have that
$$
\mathcal{D}_{\Psi,N}(\rho) = 0, 
$$
for all $\Psi$ which proves the claim.
\end{proof}

\bigskip
\section{Circular cones} \label{sec_circular}

\noindent
Here we use the properties of spherical harmonics to
prove that circular cones are admissible medium cones, utilizing the  determinant condition
of Lemma~\ref{lem_DetCond}. We also use some results on associated Legendre
polynomials, which we prove later in Section~\ref{sec_assocLegendre}.

In this and the following sections, we use the notation below for
circular cones. Compare with Definition~\ref{def_source_circ}. Recall
also the notational convention mentioned before
\eqref{eq_spherical_coords} that allows us to write
$(\vartheta,\varphi)\in\mathbb S^2$ instead of $x\in\mathbb S^2$, $x =
(\sin\vartheta \cos\varphi, \sin\vartheta\sin\varphi, \cos\vartheta)$.
\begin{definition}\label{def_circCone}
  By a circular cone $C_\rho$ we denote a cone that can be represented
  in a spherical coordinate system as
  \[
  C_\rho := \big\{ (r,\vartheta,\varphi) \in \R^3 \;:\;
  r>0,\,\vartheta \in (0,\rho),\, 0\leq\varphi<2\pi \big\},
  \]
  where $\rho \in (0,\pi/2)$. Also we let $K_\rho := C_\rho \cap
  \mathbb S^2$. Note that a half space is not a spherical cone according to
  this definition.
\end{definition}

\medskip
\noindent
Recall that the elements of the matrix $\mathcal{C}_N(\rho)$ are given by
$$
I_N^{k,l}(\rho) = \int_{K_\rho} Y_{N+2}^{k} \ov{Y}_{N}^{l} \,dS,
$$
and moreover that
$$
\mathcal{C}_{N}(\rho) :=
\begin{pmatrix}
I_N^{-N,-N}(\rho) & \dots & I_N^{-N,N}(\rho)  \\
\vdots & \ddots & \vdots \\
I_N^{N,-N}(\rho) & \dots & I_N^{N,N}(\rho) 
\end{pmatrix},
$$
Furthermore we set
$$
\mathcal{D}_{N}(\rho) = \det \mathcal{C}_{N}(\rho).
$$
The next proposition shows that circular cones are admissible medium cones,
and taken together with Theorem~\ref{thm_admScatter} 
we see that circular cones always scatter.

\begin{proposition} \label{prop_circular}
A circular cone $C_\rho$ is an admissible medium cone. 
\end{proposition}
 
\begin{proof}
By Lemma~\ref{lem_DetCond} it is clear that it is enough to show that $\mathcal{D}_{N}(\rho) \neq 0$.
In the case of a circular cone the matrix $\mathcal{C}_{N}(\rho)$ is an diagonal matrix, since
$$
k\neq l \quad \Rightarrow\quad I_N^{k,l}(\rho) = \int_{K_\rho} Y_{N+2}^{k} \ov{Y}_{N}^{l} \,dS = 0.
$$
Thus we have a matrix of the form
\begin{align*} 
\mathcal{C}_{N}(\rho) :=
\begin{pmatrix}
I_N^{-N,-N}(\rho)& \dots & 0  \\
\vdots & \ddots & \vdots \\
0  & \dots & I_N^{N,N}(\rho) 
\end{pmatrix}.
\end{align*}
Assume that $m \in \{-N, \dots N\}$.
The diagonal elements are then of the form 
\begin{align*} 
I_N^{m,m}(\rho)  
&= \int_{K_\rho} Y_{N+2}^{m} \ov{Y}_{N}^{m} \,dS  \\
&= C
\int_{0}^{2\pi} \int_0^\rho 
e^{im\varphi} P^m_{N+2}(\cos \vartheta) e^{-im\varphi} P^m_{N}(\cos \vartheta) \sin\vartheta\,d\vartheta d\varphi \\
&=
2\pi C \int_0^\rho  P^m_{N+2}( \cos \vartheta) P^{m}_{N}( \cos \vartheta) \sin\vartheta\,d\vartheta  \\
&=
2\pi C \int_{x_0}^1  P^m_{N+2}(x) P^{m}_{N}(x) \,dx  \\
&=
2\pi C' \int_{x_0}^1  P^{|m|}_{N+2}(x) P^{|m|}_{N}(x) \,dx,
\end{align*}
where $x_0 = \cos \rho$, and where we used the definition in \eqref{eq_Pmn} to deduce that
$P^{-m}_M = c P^{m}_M$, for some  constants $c$.
We will evaluate the integral on the last line. Clearly we can assume that $m\geq 0$.
By lemmas \ref{lem_integral}, \ref{lem_integral2}, \ref{lem_integral3} and 
\ref{lem_ChristoffelDarboux3}, we know that
\begin{align*} 
\int_{x_0}^1  P^m_{N+2}(x) P^{m}_{N}(x) \,dx
=
(1-x_0^2) 2a_{N+1} x_0 
\left\{
\begin{aligned}
& \sum_{j=0}^{(N-m)/2} c_j [P^m_{m+2j}(x_0)]^2, &\text{$N-m$ is even} \\
& \sum_{j=0}^{(N-m-1)/2} \tilde c_j [P^m_{m+1+2j}(x_0)]^2, &\text{$N-m$ is odd.} \\
\end{aligned}
\right.
\end{align*}
In particular we have that
\begin{align*} 
\int_{x_0}^1  P^m_{N+2}(x) P^{m}_{N}(x) \,dx
\geq
\left\{
\begin{aligned}
&C x_0(1-x^2_0)[P^m_m(x_0)]^2, &\text{ $N-m$ is even} \\
&C x_0(1-x^2_0)[P^m_{m+1}(x_0)]^2, &\text{ $N-m$ is odd}
\end{aligned}
\right.
\;\;> \;0,
\end{align*}
since $P^m_m$ and $P_{m+1}^m$ have no zeros on the interval $x_0 \in (0,1)$, and
since $x_0 \in (0,1)$, because $\rho \in (0,\pi/2)$. 
It follows that all the diagonal elements of $\mathcal{C}_{N}(\rho)$ are bounded away from zero, when 
$\rho \in (0,\pi/2)$, and hence
$$
\mathcal{D}_{N}(\rho) \neq 0.
$$
A circular cone is thus an admissible medium cone. 
\end{proof}

\section{A density argument for star-shaped cones}\label{sec_density}

\noindent
In this section we finish the proof of Theorem~\ref{thm_starScat} by
proving Proposition~\ref{prop_affine}. The latter shows that all star-shaped cones
have admissible medium cones arbitrarily close to them. Theorem~\ref{thm_admScatter} imples that these admissible ones always scatter.

Let us begin by recalling Definition~\ref{def_starCone0} of a
star-shaped cone, that states that a cone $C^\sigma$, with the vertex
at the origin, is \emph{star-shaped}, if
\[
C^\sigma \cap \mathbb S^2 := \big\{ (\vartheta,\varphi) \in \mathbb S^2 \;:\;
0\leq\varphi<2\pi,\, 0 \leq \vartheta < \sigma(\varphi) \big\},
\]
where $\sigma \colon [0,2\pi] \to (\rho_0,\pi/2)$, $\rho_0 \in (0,\pi/2)$
is a continuous, with $\sigma(0)=\sigma(2\pi)$.

In the following, we are interested in deformations of circular cones
(Definition~\ref{def_circCone}) into a star-shaped cone, which we call
star-shaped deformations.

\begin{definition} \label{def_starShaped}
A \emph{star-shaped deformation} of a circular cone $C_{\rho_0}$,
$\rho_0 \in (0,\pi/2)$, is a family of cones $\rho\mapsto
C^{\sigma}_{\rho}$ with vertex at the origin and having an
intersection of the form
\[
C^{\sigma}_{\rho} \cap \mathbb S^2 = \big\{ (\vartheta,\varphi) \in \mathbb S^2 \;:\;
\varphi \in [0,2\pi),\, 0 \leq \vartheta < \rho \sigma(\varphi) +
  (1-\rho)\rho_0 \big\}, \quad \rho \in [ -\epsilon, 1],
\]
for some continuous $\sigma \colon [0,2\pi] \to ( \rho_0,\pi/2)$ with
$\sigma(0)=\sigma(2\pi)$, and some $\epsilon > 0$ small enough that
$C^\sigma_{-\epsilon}$ is star-shaped.
\end{definition}

\noindent 
Note that $C^\sigma_\rho$ is a circular cone when $\rho=0$, in particular $C^\sigma_0 = C_{\rho_0}$.
For $\rho=1$, we have $C^\sigma_\rho = C^\sigma_1= C^\sigma$. The star-shaped
deformation is thus essentially given by an interpolation between the points in
the circular cone $C_{\rho_0}$ and the star-shaped cone $C^\sigma$.

As in the previous sections, we are interested in the functions 
$$
I_{N,\sigma}^{k,l}(\rho) : = \int_{C^{\sigma}_{\rho}\cap \mathbb S^2} Y_{N+2}^{k} \ov{Y}_{N}^{l} \,dS,
$$
this time integrated over the star-shaped deformation cap instead of a circular cap.
Furthermore recall that 
$$
\mathcal{C}_{N,\sigma}(\rho) :=
\begin{pmatrix}
I_{N,\sigma}^{-N,-N}(\rho) & \dots & I_{N,\sigma}^{-N,N}(\rho)  \\
\vdots & \ddots & \vdots \\
I_{N,\sigma}^{N,-N}(\rho) & \dots & I_{N,\sigma}^{N,N}(\rho) 
\end{pmatrix},
$$
and that
\begin{equation} \label{eq_starD}
\mathcal{D}_{N,\sigma}(\rho) := \det \mathcal{C}_{N,\sigma}(\rho).
\end{equation}

We will also be dealing with the \emph{associated Legendre
  polynomials}\footnote{which are not always polynomials!}  which are
defined in \eqref{eq_Pmn} in terms of the Legendre polynomials $P_n$,
but we reproduce the equation here for convenience.
\[
P^m_n := (-1)^m (1-x^2)^{m/2} \p_x^m P_n, \quad
P^{-m}_n := (-1)^m \frac{(n-m)!}{(n+m)!}  P^m_n, 
\]
where $m=0,..,n$. Moreover we set $P_n^m := 0$, for $m>n$.

It will be convenient to extend the definition of the associated
Legendre polynomials to $\C$. Note that extending the factor
$(1-x^2)^{1/2}$ can be done in several ways depending on which branch
of the square root we choose.

\begin{definition}\label{def_PmnC}
For $z \in \C$ we define
$$
P^m_n(z) := (-1)^m (1-z^2)^{m/2} \p_z^m P_n(z), \quad
P^{-m}_n(z) := (-1)^m \frac{(n-m)!}{(n+m)!}  P^m_n(z), 
$$
where we choose the square root so that it has the branch cut along $(-\infty, 0)$,
that is $\sqrt{z} = \sqrt{re^{i\varphi}} = \sqrt{r} e^{i\varphi/2}$, when $ \varphi \in (-\pi,\pi)$.
\end{definition}

\begin{lemma} \label{lem_Pnm_analytic}
The functions $P_n^m$ are complex analytic in $\{ x+iy\in\C\;:\; -1<x<1,\, y\in\R\}$.
\end{lemma}

\begin{proof} Since $\p_z^m P_n(z)$ is a polynomial and $(1-z)^{m/2}$ is a product, it will be enough
to show that $(1-z^2)^{1/2}$ is complex analytic on $\{x+iy\in\C\;:\;-1<x<1,\,y\in\R\}$.
Now 
$$
(1-z^2)^{1/2} = \sqrt{z} \circ (1-z) \circ z^2.
$$
Since $\Re z^2 = x^2-y^2 \in (-\infty,1)$ when $-1<x<1$, we have
\[
z^2 \colon \{x+iy\in\C\;:\;-1<x<1,\,y\in\R\} \to
\{x+iy\in\C\;:\;x<1,\,y\in\R\}
\]
is analytic and
\[
(1-z)\colon\{x+iy\in\C\;:\;x<1,\,y\in\R\} \to
\{x+iy\in\C\;:\;x>0,\,y\in\R\}
\]
is analytic. Finally $\sqrt{z} \colon \C \setminus (-\infty,0] \to \C$ is
analytic since we choose the branch cut in Definition~\ref{def_PmnC}
to be $(-\infty,0)$. The range of $(1-z)$ is contained in the former's
domain and so the restriction of $\sqrt{z}$ to the range of $(1-z)$ is
analytic too. Thus $(1-z^2)^{1/2}$ is complex analytic on
$\{x+iy\in\C\;:\;-1<x<1,\,y\in\R\}$, which proves the claim.
\end{proof}

\noindent
Next we will shows that $\mathcal{D}_{N,\sigma}$ are analytic
functions in the parameter $\rho$. For this it is convenient to make
use of the following elementary lemma. But before that, let us recall
some notation. For two sets of complex numbers $A,B\subset\C$ we
define
\[
AB = \{ab\in\C \;:\; a\in A, \, b\in B\}, \qquad A+B = \{a+b\in\C \;:\; a\in A, \, b\in B\}.
\]
This notation is particularly useful for talking about rectangles.

\begin{lemma} \label{lem_analycityHLP}
  Let $a<b$ and $c<d$ be real numbers and $\delta$ a positive real
  number. Let $g,h \in L^\infty(a,b) \cap C(a,b)$ and denote
  \[
  \mathcal{R}' := (c,d) + i(-\delta,\delta).
  \]
  Suppose that
  \begin{enumerate}[(i)]
  \item a complex valued function $f$ is complex analytic on
    $\mathcal{R} := \mathcal{R}' g[(a,b)]$, where $g[(a,b)]$ is the
    image of the interval $(a,b)$ under $g$. Assume furthermore that
  \item for any $z_0 \in \mathcal{R}'$, one has the estimate $\int_a^b
    |f( z_0 g(t) ) h(t)| \,dt < \infty$.
  \end{enumerate}
  For $z_0\in\mathcal{R}'$, let
  \[
  F(z_0) := \int_a^b f( z_0 g(t) ) h(t) \,dt.
  \]
  Then $F$ is complex analytic on $\mathcal{R}'$. In particular $F$ is
  real analytic on $(c,d)$.
\end{lemma}

\begin{proof}
Firstly notice that the first part of condition \emph{(ii)} guarantees
that $F$ is well defined on the rectangle $\mathcal{R}'$. Secondly, it
suffices to compute the derivative $F'(z_0)$ when $z_0 \in
\mathcal{R}'$ to prove the claim.  Choose a sequence $(z_m)$,
s.t. $z_m \to 0$ in $\C$. Define
$$
F^D_m(z_0) := \int_a^b \frac{f( (z_0 + z_m) g(t)) - f( z_0 g(t) )}{z_m} h(t)  \,dt
$$ and we are going to show that $F^D_m$ converges as $m\to\infty$ and
the limit will be $F'(z_0)$. Note that $z_0\in \mathcal{R}'$, which is
an open set, and since we are interested in the limit only, we may
assume that for all $m$ the whole segment from $z_0$ to $z_0 + z_m$ is
in a compact subset $\mathcal{K} \Subset \mathcal{R}'$, which depends
only on $z_0$ and $\mathcal{R}'$. This implies that
\begin{equation} \label{MVparams}
  \xi g(t) \in \mathcal{K} g[(a,b)] \Subset \mathcal{R}' g[(a,b)] =
  \mathcal{R}
\end{equation}
for any $t\in(a,b)$ and $\xi$ on that segment, i.e. that $\xi g(t)$
will stay a positive distance from $\partial\mathcal{R}$. This will be
used later in the proof after a mean value theorem.

To study the limit of $F^D_m(z)$ define
\[
D_m(z_0,t) := \frac{f( (z_0 + z_m) g(t)) - f( z_0 g(t) )}{z_m}.
\]
Let us study the limit of these as $m\to\infty$. Firstly, the
pointwise limit. If $g(t)=0$ then $D_m(z_0,t)=0$ so the pointwise
limit exists and is $0$. If $g(t)\neq0$ then the pointwise limit of
$D_m(z_0,t)$ is $f'(z_0g(t)) g(t)$ a formula which also applies to the
case $g(t)=0$. Next, let us show that the $|D_m(z_0,\cdot)|$ have an
integrable upper bound in the interval $(a,b)$. If $g(t)\neq0$, we
have
\begin{align*}
  |D_m(z_0,t)| &= \Big| \frac{f( (z_0 + z_m) g(t)) - f( z_0 g(t)
    )}{z_m g(t)} \Big| |g(t)| \\ &\lesssim |f'(\xi g(t))| |g(t)| \leq
  \sup_{\zeta\in \mathcal{K}g(a,b)} |f'(\zeta)| |g(t)|.
\end{align*}
by the mean value theorem for some $\xi\in\C$ which lies on the
segment connecting $z_0$ to $z_0+z_m$. This also holds when
$g(t)=0$. The supremum above is a finite number because $f$ is complex
analytic in $\mathcal{R}$ by \textit{(i)}, and so its derivative $f'$
is bounded on its compact subsets. The set $\mathcal{K} g(a,b)$ is
such a set according to \eqref{MVparams}. This means that for any
fixed $z_0$ there is a finite constant $C$ depending on
$a,b,c,d,\delta,z_0$ and $f$ such that $|D_m(z_0,t)| \leq C |g(t)|
\leq C \|g\|_\infty$ for all $t\in(a,b)$, and this is integrable.

By the dominated convergence Theorem and since $h\in L^\infty(a,b)$,
we have that for each $z_0\in\mathcal{R}'$ the following holds
\begin{align*}
  \lim_{m\to\infty} F_m(z_0) &= \lim_{m\to\infty} \int_a^b D_m(z_0,t)
  h(t) \,dt \\ &= \int_a^b\lim_{m\to\infty} D_m(z_0,t) h(t) \,dt \\ &=
  \int_a^b f'(z_0g(t)) g(t) h(t) \,dt
\end{align*}
and that the latter is a finite complex number. This gives the
existence of $F'(z_0)$. Hence $F$ is complex
analytic in $\mathcal{R}'$.
\end{proof}

\begin{lemma} \label{lem_analycity2}
The functions $\rho\mapsto\mathcal{D}_{N,\sigma}(\rho)$ of any
star-shaped deformation $\rho\mapsto C^\sigma_\rho$,
$\rho\in(-\epsilon,1)$ are real analytic functions in
$(-\epsilon_0,1)$ for some $0<\epsilon_0\leq\epsilon$. They are also
not identically zero.
\end{lemma}

\begin{proof}
Using spherical coordinates and referring to \eqref{YNmDef} we have that 
\begin{align*} 
I_{N,\sigma}^{k,l}(\rho) 
&= 
\int_{C^{\sigma}_{\rho}\cap \mathbb S^2} Y_{N+2}^{k} \ov{Y}_{N}^{l} \,dS \\ 
&= 
\int_0^{2\pi}\int_0^{\rho_0 + \rho(\sigma(\varphi)-\rho_0)}
Y_{N+2}^{k}(\vartheta,\varphi) \ov{Y}_{N}^{l}(\vartheta,\varphi) \sin(\vartheta) \,d\vartheta d\varphi  \\
&=
C_N \int_0^{2\pi} e^{i(k-l)\varphi} 
\int_0^{\cos(\rho_0 + \rho(\sigma(\varphi)-\rho_0))}
P_{N+2}^{k}(x) P_{N}^{l}(x) \,dx d\varphi 
\end{align*}
Changing the variable symbols in preparation for
Lemma~\ref{lem_analycityHLP}, we can write
\[
I_{N,\sigma}^{k,l}(z_0) = C_N \int_0^{2\pi} f(z_0 g(t))h(t)
\,dt
\]
where $f = f_1 \circ f_2$ and 
\begin{align*}
  f_1(z) &=  \int_0^{z} P_{N+2}^{k}(w) P_{N}^{l}(w) \,dw, \\
  f_2(z) &= \cos(\rho_0 + z),\\
  g(t) &= \sigma(t)-\rho_0, \\
  h(t) &=  e^{i(k-l)t}
\end{align*}
and the integral in the definition of $f_1$ is interpreted as the
complex line integral of the straight line segment from $0$ to $z$.

We want to apply Lemma~\ref{lem_analycityHLP} to show that
$I_{N,\sigma}^{k,l}$ is real analytic in $(-\epsilon_0,1)$ for some
$\epsilon_0>0$. In order to do this we need to check that the
conditions \emph{(i)} and \emph{(ii)} of the Lemma~\ref{lem_analycityHLP} can be satisfied. For condition \textit{(i)},
we need to find $\epsilon_0 > 0$ and $\delta_0 > 0$ so that $f$ is
complex analytic on the rectangle
\begin{equation} \label{toproveI}
  \mathcal{R} := \mathcal R' g[(0,2\pi)],
\end{equation}
where
\begin{equation}
  \mathcal R' := (-\epsilon_0,1) + i (-\delta_0,\delta_0)
\end{equation}
and $\epsilon_0\leq\epsilon$ so that $C^\sigma_\rho$ is defined for
$\rho\in(-\epsilon_0,1)$. Choose $\epsilon_0 > 0$ by
\begin{equation} \label{epsilonChoice}
  \epsilon_0 := \min\Big(\frac{\rho_0}{2\|g\|_{L^\infty}},
  \epsilon\Big).
\end{equation}
Since $g(t) \in (0,\pi/2-\rho_0)$ for all $t$, and by our choice of
$\epsilon_0$, we have for $x_0 \in (-\epsilon_0, 1)$, that
\begin{align*} 
  \rho_0/2 < \rho_0 + x_0 g(t) < \pi/2.
\end{align*}
Recal also that
\[
\cos(x+iy) = \cos(x)\cosh(y) - i \sin(x)\sinh(y),
\]
where each of $\cos,\cosh,\sin,\sinh$ are real-valued when their
arguments are real. Let $\delta_2>0$ be arbitrary and let
\begin{equation} \label{deltaChoice}
  \delta_0 := \frac{\delta_1}{\|g\|_{L^\infty}}, \qquad \delta_1 =
  \frac12 \min\Big(\arsinh\delta_2,\,
  \arcosh\Big(\frac{1}{\cos(\rho_0/2)}\Big)\Big)
\end{equation}
both of which are positive, and they give $\delta_0 g(t) \in
(0,\delta_1)$ for all $t$. Then for $-\epsilon_0<x_0<1$ and $-\delta_0
< y_0 < \delta_0$, we have that
\begin{align*}
&f_2((x_0 + i y_0)g(t)) = \cos(\rho_0 + (x_0 + i y_0)g(t)) 
  \\ &\qquad= \cos(\rho_0 + x_0 g(t)) \cosh(y_0 g(t)) - i \sin(\rho_0 + x_0 g(t)) \sinh(y_0 g(t))
  \\ &\qquad \subset \cos[(\rho_0/2,\pi/2)]\cosh[(0,\delta_1)] - i \sin[(\rho_0/2,\pi/2)]\sinh[(-\delta_1,\delta_1)]
  \\ &\qquad \subset (0, \cos(\rho_0/2)\cosh \delta_1) + i (-\sinh\delta_1, \sinh\delta_1).
\end{align*}
Denote $\alpha=0$, $\beta=\cos(\rho_0/2)\cosh\delta_1$. Then by our choice of
$\delta_1$ we see that $\alpha<\beta<1$ and $\sinh\delta_1 < \delta_2$, and so
\[
f_2((x_0 + i y_0)g(t)) \subset (\alpha,\beta) + i (-\delta_2,\delta_2) =:
\mathcal R''
\]
for $x_0 \in (-\epsilon_0, 1)$, $y_0\in(-\delta_0,\delta_0)$ and
$t\in(0,2\pi)$. In other words,
\begin{equation} \label{f2anal}
  f_2 \colon \mathcal R \to \mathcal R''
\end{equation}
is complex analytic. In particular the associated Legendre polynomials
$P^m_n$ are analytic in the range of $f_2$. In fact, $\mathcal R''$ is
a positive distance $1-\beta$ away from any point where we do not know
them being analytic. See Lemma~\ref{lem_Pnm_analytic}.

Let us prove next that $f$ is complex analytic on the rectangle $\mathcal{R}$ given in \eqref{toproveI}. 
We know that $P_{N+2}^kP^l_N$ is
analytic on the vertical strip $\{x+iy\in\C\;:\;-1<x<1,\,y\in\R\}$ by
Lemma~\ref{lem_Pnm_analytic}. Hence
\[
f_1\colon z\mapsto \int_0^{z} P_{N+2}^{k}(w) P_{N}^{l}(w) \,dw
\]
is likewise analytic in the strip $\{x+iy\in\C\;:\;-1<x<1,\,y\in\R\}$, which happens to contain
$\mathcal R''$. Thus $f=f_1\circ f_2$ is analytic on $\mathcal R$ and so
condition \emph{(i)} of Lemma~\ref{lem_analycityHLP} is satisfied.

Next, we show that condition \emph{(ii)} is satisfied, i.e. that we
have $\int_0^{2\pi} |f(z_0g(t))h(t)| dt < \infty$ for any
$z_0\in\mathcal R'$. Recall that $|h(t)|=1$ in our case. Notice that
$f_1$ is in fact uniformly bounded on $\mathcal R''$. This is because
$f_1$ is analytic in the vertical strip mentioned above, and $\mathcal
R''$ is contained in it with distance $1-\beta>0$ from its
boundary. Since $f_2\colon\mathcal R \to \mathcal R''$, and $f_1\colon\mathcal
R''\to\C$ is uniformly bounded, we have $f = f_1\circ f_2 \colon \mathcal
R\to\C$ uniformly bounded. Given $z_0\in\mathcal R'$ and
$t\in(0,2\pi)$ we have $z_0g(t)\in\mathcal R$ and so we have a uniform
bound for $f(z_0g(t))$. Hence the integral is finite.

Lemma~\ref{lem_analycityHLP} imples that the maps $\rho \mapsto
I_{N,\sigma}^{k,l}(\rho)$ are real analytic on $(-\epsilon_0,1)$ where
$\epsilon_0$ is given by \eqref{epsilonChoice}. The determinant
functions $\mathcal{D}_{N,\sigma}$ are obtained from matrices having
elements $I_{N,\sigma}^{k,l}$ as their components. The determinant is
a sum of products of these elements and is therefore also real
analytic in $(-\epsilon_0, 1)$.

Finally, by the proof of Proposition~\ref{prop_circular} we know that
$\mathcal{D}_{N,\sigma}(0) \neq 0$, since $C^\sigma_0$ is a circular
cone. The functions $\mathcal{D}_{N,\sigma}$ are thus not identically
zero in $\rho$. This completes the proof.
\end{proof}

\noindent
We will now show that there exists a number of star-shaped 
cones, which are not necessarily circular, but still admissible medium cones.

\begin{proposition} \label{prop_affine}
Let $\rho\mapsto C_\rho^{\sigma}$, $\rho\in(-\epsilon,1)$, be a
star-shaped deformation which satisfies the regularity assumptions
\textit{(i)}--\textit{(iii)} of
Definition~\ref{def_admissible-medium-corner}. Let
$0<\epsilon_0<\epsilon$ be as in Lemma~\ref{lem_analycity2}. Then
there is a countable set $\mathcal Z$ such that all $C^\sigma_\rho$
with $\rho\in(-\epsilon_0,1)\setminus\mathcal Z$ are admisible.
\end{proposition}
\begin{proof}
By Lemma~\ref{lem_analycity2} we know that for every $N$ the map
$\rho\mapsto\mathcal{D}_{N,\sigma}(\rho)$ is analytic on
$(-\epsilon_0,1)$ without being identically zero. Consider the set
\[
\mathcal{Z} = \big\{ \rho \in (-\epsilon_0,1) \;:\;
\mathcal{D}_{N,\sigma}(\rho) = 0 \text{ for some } N \big\}.
\]
Lemma~\ref{lem_DetCond} implies that all of the cones
$C^\sigma_{\rho}$ with $\rho\in(-\epsilon_0,1)\setminus\mathcal Z$ are
admissible medium cones because $\rho\in
(-\epsilon_0,1)\setminus\mathcal Z$ implies that
$\mathcal{D}_{N,\sigma}(\rho) \neq 0$ for all $N$.

The set $\mathcal{Z}$ is countable as a countable union of countable
sets. This follows because $\mathcal D_{N,\sigma}$ has at most a
countable number of zeros in $(-\epsilon_0,1)$ since it is analytic
and not identically zero there by Lemma~\ref{lem_analycity2}.
\end{proof}

\section{Results on associated Legendre polynomials} \label{sec_assocLegendre}

\noindent
In this section we derive some results on associated Legendre polynomials, that are utilized 
in computing explicitly the determinants $\mathcal{D}_N$ in the case
of circular cones, which is done in Section~\ref{sec_circular}.
First we extend some well known formulas for certain inner products
of Legendre polynomials to the case of the associated Legendre polynomials,
after which we derive a modification of the Christoffel-Darboux formula,
which can be used to analyze these inner products.

\medskip
\noindent
Recall firstly that the associated Legendre polynomials are defined in
terms of the Legendre polynomials $P_n$, as
\begin{equation} \label{eq_Pmn}
P^m_n := (-1)^m (1-x^2)^{m/2} \p_x^m P_n, \quad
P^{-m}_n := (-1)^m \frac{(n-m)!}{(n+m)!}  P^m_n, 
\end{equation}
where $m=0,..,n$. Moreover we set $P_n^m := 0$, for $m>n$. The Legendre polynomials are defined in \cite{byerly93_fourier} equation (9), p.10.

We will first derive an integral formula for the special case of the associated Legendre polynomials
of the form $P^0_n$, i.e. in the case when they coincide with the Legendre polynomials $P_n$.
This formula can be found in \cite{byerly93_fourier} equation (5), p.172.
We however give the proof as a convenience to the reader.

\begin{lemma} \label{lem_integral}
Let $x_0 \in (-1,1)$. Then we  have the following formula
$$
\int_{x_0}^1 P^0_n P^0_{n+2}  \,dx = C(1-x_0^2)\big( \p_x P^0_{n+2} P^0_n - P^0_{n+2}  \p_x P^0_n \big)\big|_{x=x_0},
$$
\end{lemma}
where $C = 1/(4n+6)$.
\begin{proof} 
Note firstly that $P_n^0 = P_n$.
The Legendre polynomials in the claim solve the ODEs
\begin{align*} 
\p_x \big((1-x^2) \p_x P_n \big) + n(n+1)P_n &=0, \\
\p_x \big((1-x^2) \p_x P_{n+2} \big) + (n+2)(n+3)P_{n+2} &=0.
\end{align*}
Multiplying by $P_n$ and $P_{n+2}$, and subtracting the resulting equations gives, then
that
\begin{align*} 
-(4n +6) \int_{x_0}^1 P_n P_{n+2}  \,dx 
&=\int_{x_0}^1  \p_x \big((1-x^2) \p_x P_n \big)P_{n+2} - \p_x \big((1-x^2) \p_x P_{n+2} \big)P_n  \,dx  \\
&= (1-x^2) \big( \p_x P_n  P_{n+2} -  P_{n+2} \p_x P_n \big) \Big|_{x_0}^1 \\ 
&= (1-x_0^2) \big(P_{n+2}(x_0)  \p_x P_n(x_0) - \p_x P_{n+2}(x_0) P_n(x_0) ),
\end{align*}
which proves the claim.
\end{proof}

\medskip
\noindent
Now we generalize the  formula of Lemma~\ref{lem_integral} to the case where $m=1,..,n-1$.
First we derive the following simple formula (which can be found in e.g. \cite{byerly93_fourier}
p. 205).

\begin{lemma} \label{lem_ODEhlp}
For $m \geq 1$, we have the formula
$$
\p_x \big((1-x^2)^m\p_x^{m}P_n \big) = - C (1-x^2)^{m-1}\p_x^{m-1}P_n,
$$
where $C=  n(n+1) - m(m-1)$.
\end{lemma}

\begin{proof} 
By taking repeated derivatives of the Legendre ODE
$$
(1-x^2)\p_x^2P_n - 2x\p_x P_n + n(n+1)P_n = 0,
$$
one arrives by an induction argument at the formula
$$
(1-x^2)\p_x^{m+1}P_n - 2mx\p^m_x P_n + (n(n+1)-m(m-1)) \p_x^{m-1}P_n = 0.
$$
Multiplying by $(1-x^2)^{m-1}$, gives then that
$$
\p_x \big((1-x^2)^m\p_x^{m}P_n \big) = - (n(n+1)- m (m-1)) (1-x^2)^{m-1}\p_x^{m-1}P_n.
$$
\end{proof}

\noindent
We are now ready to generalize the result of Lemma~\ref{lem_integral} to the case $m=1,..,n-1$.

\begin{lemma} \label{lem_integral2}
Let $x_0 \in (-1,1)$ and $m=1,..,n-1$. Then we  have the following formula
\begin{align*} 
\int_{x_0}^1 P^m_n P^m_{n+2}  \,dx =  C (1-x_0^2) \big( \p_x P^m_{n+2} P^m_n - P^m_{n+2}  \p_x P^m_n \big)\big|_{x=x_0},
\end{align*}
where $C = 1/(4n+6)$.
\end{lemma}

\begin{proof}
Assume that $1 \leq m \leq n-1$.
By integration by parts we get that
\begin{align*} 
\int_{x_0}^1 P^m_n P^m_{n+2}  \,dx 
&= 
\int_{x_0}^1  (1-x^2)^m \p_x^m P_{n+2} \p_x^m P_n\,dx \\
&= 
-\int_{x_0}^1  \p_x ((1-x^2)^m \p_x^m P_{n+2}) \p_x^{m-1} P_n\,dx 
+ \Big|_{x_0}^1 (1-x^2)^m \p_x^m P_{n+2} \p_x^{m-1} P_n. 
\end{align*}
Then using the formula of Lemma~\ref{lem_ODEhlp} we have that
\begin{align*} 
\int_{x_0}^1 P^m_n P^m_{n+2}  \,dx 
&= 
C_1 \int_{x_0}^1  (1-x^2)^{m-1} \p_x^{m-1} P_{n+2} \p_x^{m-1} P_n\,dx 
+ \Big|_{x_0}^1 (1-x^2)^m \p_x^m P_{n+2} \p_x^{m-1} P_n \\
&= 
C_1 \int_{x_0}^1 P^{m-1}_{n+2} P^{m-1}_n\,dx 
+ \Big|_{x_0}^1 (1-x^2)^m \p_x^m P_{n+2} \p_x^{m-1} P_n,
\end{align*}
where $C_1 =  (n+2)(n+3) - m(m-1)$.
Likewise we can use integration by parts to deduce that
\begin{align*} 
\int_{x_0}^1 P^m_n P^m_{n+2}  \,dx 
&= 
C_2 \int_{x_0}^1 P^{m-1}_{n+2} P^{m-1}_n\,dx 
+ \Big|_{x_0}^1 (1-x^2)^m \p_x^{m-1} P_{n+2} \p_x^{m} P_n.
\end{align*}
where $C_2=  n(n+1) - m(m-1)$. 
Subtracting these we obtain that
\begin{align*} 
C_3 \int_{x_0}^1 P^{m-1}_{n+2} P^{m-1}_n\,dx 
&= 
\Big|_{x_0}^1 \Big( (1-x^2)^m \p_x^{m-1} P_{n} \p_x^{m} P_{n+2}
-
(1-x^2)^m \p_x^{m} P_{n} \p_x^{m-1} P_{n+2} \Big),
\end{align*}
where the constant is given by 
$
C_3 := C_1- C_2 = 4n+6.
$
This can be further simplified, since by a straight forward computation
\begin{align*} 
(1-x^2)^m \p_x^{m-1} P_{n+2} \p_x^{m} P_n
=
(1-x^2) P^{m-1}_{n+2} \p_x P^{m-1}_n + (m-1)xP_{n+2}^{m-1}P_n^{m-1},
\end{align*}
and likewise 
\begin{align*} 
(1-x^2)^m \p_x^{m} P_{n+2} \p_x^{m-1} P_n
=
(1-x^2) \p_x P^{m-1}_{n+2} P^{m-1}_n + (m-1)xP_{n+2}^{m-1}P_n^{m-1}.
\end{align*}
Putting this together gives then that
\begin{align*} 
C_3 \int_{x_0}^1 P^{m-1}_{n+2} P^{m-1}_n\,dx 
&= 
\Big|_{x_0}^1 (1-x^2)\Big( P^{m-1}_{n} \p_x P^{m-1}_{n+2}
-
\p_x P^{m-1}_{n} P^{m-1}_{n+2} \Big).
\end{align*}
This proves the claim for the cases $m = 1,..,n-1$. 
\end{proof}

\noindent
Finally we prove the results of lemmas \ref{lem_integral} and \ref{lem_integral2}, for the case $m=n$.

\begin{lemma} \label{lem_integral3}
Let $x_0 \in (-1,1)$. Then we  have the following formula
\begin{align*} 
\int_{x_0}^1 P^n_n P^n_{n+2}  \,dx =  C (1-x_0^2) \big( \p_x P^n_{n+2} P^n_n - P^n_{n+2}  \p_x P^n_n \big)\big|_{x=x_0},
\end{align*}
where $C = 1/(4n+6)$.
\end{lemma}

\begin{proof}
Note that $\p_x^{n+1} P_n = 0$. By integration by parts we have that
\begin{equation} \label{eq_parts}
\begin{aligned}
0&=  \int_{x_0}^1  (1-x^2)^{n+1} \p_x^{n+1} P_{n+2} \p_x^{n+1} P_n\,dx \\
&= 
-\int_{x_0}^1  \p_x ((1-x^2)^{n+1} \p_x^{n+1} P_{n+2}) \p_x^{n} P_n\,dx 
+ \Big|_{x_0}^1 (1-x^2)^{n+1} \p_x^{n+1} P_{n+2} \p_x^{n} P_n. 
\end{aligned}
\end{equation}
Let's rewrite the first term on the r.h.s. of the last line.
Using Lemma \eqref{lem_ODEhlp} we have that
\begin{align*} 
\p_x ((1-x^2)^{n+1} \p_x^{n+1} P_{n+2}) = -(4n+6) (1-x^2)^n \p_x^n P_{n+2}.
\end{align*}
Notice that by inserting this back into to \eqref{eq_parts} we get the term on the l.h.s. 
of the claim.

It is thus enough to show that the boundary term in \eqref{eq_parts} gives the expression on the r.h.s.
of the claim. For this  we use the fact that 
\begin{align*} 
(1-x^2)^{n+1} \p_x^{n+1} P_{n+2} \p_x^{n} P_n
&=
(1-x^2)  P_n^n \big( (1-x^2)^{n/2} \p_x^{n+1} P_{n+2} \big) \\
&=
(1-x^2)  P_n^n \big( \p_x P^n_{n+2} -  \p_x (1-x^2)^{n/2} \p_x^{n} P_{n+2} \big).
\end{align*}
Now we have that
$$
\p_x P_n^n = \p_x (1-x^2)^{n/2} \p_x^{n} P_{n}, 
$$
so the previous equations yields by a short computation that 
\begin{align*} 
(1-x^2)^{n+1} \p_x^{n+1} P_{n+2} \p_x^{n} P_n
&=
(1-x^2)  \big( P_n^n\p_x P^n_{n+2} -  P^n_{n+2} \p_x P^n_n \big).
\end{align*}
Going back to \eqref{eq_parts} we see that
\begin{equation*} 
\begin{aligned}
(4n+6)\int_{x_0}^1 (1-x^2)^n \p_x^n P_{n+2}  \p_x^{n} P_n\,dx 
&=
-\Big|_{x_0}^1 
(1-x^2)  \big( P_n^n\p_x P^n_{n+2} - P^n_{n+2} \p_x P^n_n \big),
\end{aligned}
\end{equation*}
and that the claim holds.
\end{proof}

\noindent
Next we prove the following Christoffel-Darboux type formula, which is a slight modification
of the usual formula (see p.43 in \cite{szego75_orthog_polyn}). Note that the r.h.s. of the claim
has a sign determined by the sign of $x$.

\begin{lemma} \label{lem_ChristoffelDarboux3}
Let $0\leq m \leq n$. We have the  following identity
\begin{align*}
P^m_{n} \p_x P^m_{n+2}  -  \p_x P^m_{n} P^m_{n+2} 
=  2 a_{n+1}x
\left\{
\begin{aligned}
& \sum_{j=0}^{(n-m)/2} c_j [P^m_{m+2j}]^2, &\text{$n-m$ is even} \\
& \sum_{j=0}^{(n-m-1)/2} \tilde c_j [P^m_{m+1+2j}]^2, &\text{$n-m$ is odd} \\
\end{aligned}
\right.
\end{align*}
where the coefficients $c_j$ are given by
\begin{equation*} 
c_{j} :=
\left\{
\begin{aligned}
& a_n, &j=(n-m)/2, \\
& a_{m+2j}b_{n-1}b_{n-2}\dots b_{m+2j}, &j < (n-m)/2.
\end{aligned}
\right.
\end{equation*}
And the coefficients $ \tilde c_j$ are given by
\begin{equation*} 
\tilde c_{j} :=
\left\{
\begin{aligned}
& a_n, &j=(n-m-1)/2, \\
& a_{m+1+2j}b_{n-1}b_{n-2}\dots b_{m+1+2j}, &j < (n-m-1)/2.
\end{aligned}
\right.
\end{equation*}
And where coefficients $a_k$ and $b_k$ are given by
\begin{align*} 
a_k := \frac{2k+1}{k-m+1}, \quad b_k:=\frac{k+m+1}{k-m+2},  \quad \text{for} \quad m \leq k,
\end{align*}
and $a_k = 0 = b_k$, when $m > k$.

\end{lemma}

\begin{proof} 
The recurrence relation
$$
(n-m+1)P^m_{n+1} =  (2n+1)x P^m_{n} - (n+m)P^m_{n-1}
$$
implies  that 
\begin{equation} \label{eq_recurence}
\begin{aligned}
P^m_{n+1} &=  a_n x P^m_{n} -  b_{n-1} P^m_{n-1}, \\
P^m_{n+2} &= a_{n+1} x P^m_{n+1} - b_n P^m_n.
\end{aligned}
\end{equation}
Using the second of these equalities gives that
\begin{align}  \label{eq_Pcommutator}
P^m_{n}(x) P^m_{n+2}(y)  -P^m_{n}(y) P^m_{n+2}(x) 
= 
a_{n+1} 
\big(yP^m_{n+1}(y) P^m_{n}(x)  - x P^m_{n+1}(x) P^m_{n}(y) \big).
\end{align}
Lets evaluate the two terms on the r.h.s. of the  equality.
By successively applying the recurrence relation in \eqref{eq_recurence}
we get the equations
\begin{align*} 
P^m_{n+1}(y) P^m_{n}(x) &=   a_n     y P^m_{n}(y) P^m_{n}(x)     -  b_{n-1} P^m_{n-1}(y)P^m_{n}(x), \\ 
P^m_{n-1}(y) P^m_{n}(x) &=   a_{n-1} x P^m_{n-1}(x) P^m_{n-1}(y) -  b_{n-2} P^m_{n-2}(x)P^m_{n-1}(y), \\ 
P^m_{n-1}(y) P^m_{n-2}(x) &= a_{n-2} y P^m_{n-2}(x) P^m_{n-2}(y) -  b_{n-3} P^m_{n-2}(x)P^m_{n-3}(y), \\ 
P^m_{n-3}(y) P^m_{n-2}(x) &= a_{n-3} x P^m_{n-3}(x) P^m_{n-3}(y) -  b_{n-4} P^m_{n-4}(x)P^m_{n-3}(y), \\ 
& \; \; \vdots 
\end{align*}
where the last non-zero line is
\begin{equation*} 
\left\{
\begin{aligned}
&P^m_{m+1}(y) P^m_{m}(x) = a_{m} y P^m_{m}(x) P^m_{m}(y), \quad\text{if $n-m$ is even}, \\
&P^m_{m}(y) P^m_{m+1}(x) = a_{m} x P^m_{m}(x) P^m_{m}(y), \quad\text{if $n-m$ is odd}, 
\end{aligned}
\right.
\end{equation*}
since we defined $P^k_l = 0$, when $k>l$. It will be convenient to define 
\begin{equation*} 
z_1 :=
\left\{
\begin{aligned}
-&y, \text{ if $n-m$ is even}, \\
&x, \text{ if $n-m$ is odd}.
\end{aligned}
\right.
\end{equation*}
Using these equations to rewrite the first term on the r.h.s of \eqref{eq_Pcommutator},
we get that
\begin{align*} 
yP^m_{n+1}(y) P^m_{n}(x)  
= 
&a_n     y^2 P^m_{n}(y) P^m_{n}(x)  \\   
&-a_{n-1} b_{n-1}  xy P^m_{n-1}(x) P^m_{n-1}(y)  \\
&+a_{n-2} b_{n-1} b_{n-2} y^2 P^m_{n-2}(x) P^m_{n-2}(y) \\ 
&-a_{n-3} b_{n-1} b_{n-2} b_{n-3} xy  P^m_{n-3}(x) P^m_{n-3}(y) \\
& \;\;\vdots \\
&-a_{m} b_{n-1} \dots b_{m} z_1y P^m_m(y)P^m_m(x).
\end{align*}
For the second term on the r.h.s of \eqref{eq_Pcommutator}, we get similarly the expression
\begin{align*} 
xP^m_{n+1}(x) P^m_{n}(y)  
= 
&a_n     x^2 P^m_{n}(x) P^m_{n}(y)  \\   
&-a_{n-1} b_{n-1}  yx P^m_{n-1}(y) P^m_{n-1}(x)  \\
&+a_{n-2} b_{n-1} b_{n-2} x^2 P^m_{n-2}(y) P^m_{n-2}(x) \\ 
&-a_{n-3} b_{n-1} b_{n-2} b_{n-3} yx  P^m_{n-3}(y) P^m_{n-3}(x) \\
& \;\;\vdots \\
&-a_{m} b_{n-1} \dots b_{m} z_2x P^m_m(x)P^m_m(y).
\end{align*}
where the $z_2$  in the last term is defined as 
\begin{equation*} 
z_2 :=
\left\{
\begin{aligned}
-&x, \text{ if $n-m$ is even}, \\
&y, \text{ if $n-m$ is odd}.
\end{aligned}
\right.
\end{equation*}
Subtracting the two equalities gives then that
\begin{align*} 
yP^m_{n+1}(y) P^m_{n}(x)  
-xP^m_{n+1}(x) P^m_{n}(y)  
= 
&a_n (y^2-x^2) P^m_{n}(y) P^m_{n}(x)  \\   
&+a_{n-2} b_{n-1} b_{n-2} (y^2 -x^2)  P^m_{n-2}(x) P^m_{n-2}(y) \\ 
&+ \dots \\
&+a_{m} b_{n-1} \dots b_{m} ( xz_2- yz_1)  P^m_{m}(x) P^m_{m}(y). 
\end{align*}
Notice that the last term is non zero only if $n-m$ is even, and in this 
case it equals to $y^2-x^2$.
Going back to \eqref{eq_Pcommutator}, we get that
\begin{align*} 
\frac{ P^m_{n}(x) P^m_{n+2}(y)  -  P^m_{n}(y) P^m_{n+2}(x) }{x-y}
= 
-a_{n+1} \Big( &a_n (y + x) P^m_{n}(y) P^m_{n}(x)  \\   
&+a_{n-2} b_{n-1} b_{n-2} (y + x)  P^m_{n-2}(x) P^m_{n-2}(y) + \dots \Big)
\end{align*}

Adding $\pm P^m_{n+2}(y)P^m_n(y)$ in the denominator and taking the limit $y  \to x$, gives
that
\begin{align*}
P^m_{n} \p_x P^m_{n+2}  -  \p_x P^m_{n} P^m_{n+2} 
=  2 a_{n+1}x
\left\{
\begin{aligned}
& \sum_{j=0}^{(n-m)/2} c_j [P^m_{m+2j}]^2, &\text{ if $n-m$ is even} \\
& \sum_{j=0}^{(n-m-1)/2} \tilde c_j [P^m_{m+1+2j}]^2, &\text{ if $n-m$ is odd} \\
\end{aligned}
\right.
\end{align*}
where the coefficients $c_j$ are given by
\begin{equation*} 
c_{j} :=
\left\{
\begin{aligned}
& a_n, &\text{if $j=(n-m)/2$}, \\
& a_{m+2j}b_{n-1}b_{n-2}\dots b_{m+2j}, &\text{ if $j < (n-m)/2$ }.
\end{aligned}
\right.
\end{equation*}
and the coefficients $ \tilde c_j$ are given by
\begin{equation*} 
\tilde c_{j} :=
\left\{
\begin{aligned}
& a_n, &\text{if $j=(n-m-1)/2$}, \\
& a_{m+1+2j}b_{n-1}b_{n-2}\dots b_{m+1+2j}, &\text{ if  $j < (n-m-1)/2$ }.
\end{aligned}
\right.
\end{equation*}
\end{proof}

\section*{Acknowledgements}
We would like to thank Esa~V.~Vesalainen for
several ideas and computations, and without whom this work would not have been possible.
The first author was supported by the Academy of Finland project
312124 and partly by the Estonian Research Council's grant PRG 832.
The second author was supported by the grant PGC2018-094528-B-I00.


\footnotesize
\begin{thebibliography}{99}\setlength{\itemsep}{0pt}

\bibitem{agmon76_asymp_proper_solut_differ_equat} S. Agmon,
  L. H\"ormander, \textit{Asymptotic properties of solutions of
    differential equations with simple characteristics}, Journal
  d'Analyse Math\-ematique, 30(1), (1976)
  1--38. \\\url{http://dx.doi.org/10.1007/bf02786703}

\bibitem{blaasten18_nonrad_sourc_trans_eigen_vanis} E. Bl{\r a}sten,
  \textit{Nonradiating sources and transmission eigenfunctions vanish
    at corners and edges}, SIAM Journal on Mathematical Analysis,
  50(6), 2018.
  6255--6270. \\\url{http://dx.doi.org/10.1137/18m1182048}

\bibitem{blaasten18_radiat_non_radiat_sourc_elast} E. Bl{\r a}sten,
  Y. Lin, \textit{Radiating and non-radiating sources in elasticity},
  Inverse Problems, 35(1), (2018),
  015005. \\\url{http://dx.doi.org/10.1088/1361-6420/aae99e}

\bibitem{blaasten17_vanis_near_corner_trans_eigen} E. Bl{\r a}sten,
  H. Liu, \textit{On vanishing near corners of transmission
    eigenfunctions}, Journal of Functional Analysis, 273(11), (2017),
  3616--3632. \\\url{http://dx.doi.org/10.1016/j.jfa.2017.08.023}

\bibitem{blaasten17_adden_to} E. Bl{\r a}sten, H. Liu,
  \textit{Addendum to: ``On vanishing near corners of transmission
    eigenfunctions''}, arXiv e-prints:1710.08089,
  (2017). \\\url{http://arxiv.org/abs/1710.08089}

\bibitem{blaasten21_scatt_by_curvat_radiat_sourc} E. L. K. Bl{\r
  a}sten, H. Liu, \textit{Scattering by curvatures, radiationless
  sources, transmission eigenfunctions, and inverse scattering
  problems}, SIAM Journal on Mathematical Analysis, 53(4), (2021),
  3801--3837. \\\url{http://dx.doi.org/10.1137/20m1384002}

\bibitem{blaasten21_elect_probl_corner_its_applic} E. Bl{\r a}sten,
  H. Liu, J. Xiao, \textit{On an electromagnetic problem in a corner
    and its applications}, Analysis \& {PDE}, to appear (2021).

\bibitem{blaasten14_corner_alway_scatt} E. Bl{\r a}sten,
  L. P\"aiv\"arinta, J. Sylvester, \textit{Corners Always
    Scatter}. Communications in Mathematical Physics, 331(2), (2014)
  725--753. \\\url{http://dx.doi.org/10.1007/s00220-014-2030-0}

\bibitem{blaasten20_non_scatt_energ_trans_eigen_h} E. Bl{\r a}sten,
  E. V. Vesalainen, \textit{Non-scattering energies and transmission
    eigenvalues in $H^n$}, Annales Academiae Scientiarum Fennicae
  Mathematica, 45(1), (2020)
  547--576. \\\url{http://dx.doi.org/10.5186/aasfm.2020.4522}

\bibitem{boyd73_resid_set_dimen_apoll_packin} D. W. Boyd, \textit{The
  residual set dimension of the apollonian packing}, Mathematika,
  20(2), (1973)
  170--174. \\\url{http://dx.doi.org/10.1112/s0025579300004745}

\bibitem{byerly93_fourier} W. E. Byerly, \textit{An elementary
  treatise on Fourier's series and spherical, cylindrical, and
  ellipsoidal harmonics, with applications to problems in mathematical
  physics}, Boston: Ginn \& Company,
  (1893). \\\url{https://catalog.hathitrust.org/Record/000621162}

\bibitem{cakoni21_trans_eigen} F. Cakoni, D. Colton, H. Haddar,
  \textit{Transmission Eigenvalues}. Notices of the American
  Mathematical Society, 68 (9),
  2021. \\\url{http://dx.doi.org/10.1090/noti2350}

\bibitem{cakoni13_trans_eigen_inver_scatt_theor} F. Cakoni, H. Haddar,
  \textit{Transmission eigenvalues in inverse scattering theory}. In,
  Inverse Problems and Applications: Inside Out. II (pp. 529--580),
  Cambridge Univ. Press, Cambridge, (2013).

\bibitem{cakoni21_singul_almos_alway_scatt} F. Cakoni, M. S. Vogelius,
  \textit{Singularities almost always scatter: regularity results for
    non-scattering inhomogeneities}, arXiv e-prints:2104.05058,
  (2021). \\\url{https://arxiv.org/2104.05058}

\bibitem{cakoni21_corner_scatt_operat_diver_form} F. Cakoni, J. Xiao,
  \textit{On corner scattering for operators of divergence form and
    applications to inverse scattering}, Communications in Partial
  Differential Equations, 46(3), (2021),
  413--441. \\\url{http://dx.doi.org/10.1080/03605302.2020.1843489}

\bibitem{colton96_simpl_method_solvin_inver_scatt} D. Colton,
  A. Kirsch, \textit{A simple method for solving inverse scattering
    problems in the resonance region}, Inverse Problems, 12(4),
  (1996),
  383--393. \\\url{http://dx.doi.org/10.1088/0266-5611/12/4/003}

\bibitem{colton19_inver_acous_elect_scatt_theor} D. Colton, R. Kress,
  \textit{Inverse acoustic and electromagnetic scattering theory},
  Springer, Berlin, (2019).

\bibitem{elschner15_corner_edges_alway_scatt} J. Elschner, G. Hu,
  \textit{Corners and edges always scatter}, Inverse Problems, 31(1),
  (2015),
  015003. \\\url{http://dx.doi.org/10.1088/0266-5611/31/1/015003}

\bibitem{elschner17_acous_scatt_from_corner_edges_circul_cones}
  J. Elschner, G. Hu, \textit{Acoustic scattering from corners, edges
    and circular cones}, Archive for Rational Mechanics and Analysis,
  228(2), (2018),
  653--690. \\\url{http://dx.doi.org/10.1007/s00205-017-1202-4}

\bibitem{falconer03_fract} K. Falconer, \textit{Fractal geometry:
  Mathematical foundations and applications}, John Wiley \& Sons,
  Inc., Hoboken, NJ, (2003).

\bibitem{faraco12_sobol_norm_charac_funct_with} D. Faraco,
  K. M. Rogers, \textit{The Sobolev norm of characteristic functions
    with applications to the Calder\'on inverse problem}, The
  Quarterly Journal of Mathematics, 64(1), (2012),
  133--147. \\\url{http://dx.doi.org/10.1093/qmath/har039}

\bibitem{gallier13_notes_spher_harmon_linear_repres_lie_group}
  J. Gallier, \textit{Notes on spherical harmonics and linear
    representations of lie groups}, Lecture notes, Department of
  Computer and Information Science, University of Pennsylvania,
  (2013). \\\url{https://www.cis.upenn.edu/~cis610/sharmonics.pdf}

\bibitem{grafakos04_class_fourier} L. Grafakos, \textit{Classical and
  modern Fourier analysis}, Pearson Education, Inc., Upper Saddle
  River, N.J., (2004)

\bibitem{grinevich95_trans_poten_at_fixed_energ_dimen_two}
  P. G. Grinevich, R. G. Novikov, \textit{Transparent potentials at
    fixed energy in dimension two. Fixed-energy dispersion relations
    for the fast decaying potentials}, Communications in Mathematical
  Physics, 174(2), (1995),
  409--446. \\\url{http://dx.doi.org/10.1007/bf02099609}

\bibitem{hobson55} E. W. Hobson, \textit{The theory of spherical and
  ellipsoidal harmonics}, Cambridge University Press, (1931).

\bibitem{hoermander05_analy_linear_partial_differ_operat_ii}
  L. H\"ormander, \textit{The analysis of linear partial differential
    operators II: Differential Operators with Constant Coefficients},
  Springer-Verlag, Berlin, (2005).

\bibitem{hu16_shape_ident_inver_medium_scatt} G. Hu, M. Salo,
  E. V. Vesalainen, \textit{Shape identification in inverse medium
    scattering problems with a single far-field pattern}, SIAM Journal
  on Mathematical Analysis, 48(1), (2016),
  152--165. \\\url{http://dx.doi.org/10.1137/15m1032958}

\bibitem{kenig87_unifor_sobol_inequal_unique_contin} C. E. Kenig,
  A. Ruiz, C. D. Sogge, \textit{Uniform Sobolev inequalities and
    unique continuation for second order constant coefficient
    differential operators}, Duke Mathematical Journal, 55(2), (1987),
  329--347. \\\url{http://dx.doi.org/10.1215/s0012-7094-87-05518-9}

\bibitem{kirsch98_charac_shape_scatt_obstac_using} A. Kirsch,
  \textit{Characterization of the shape of a scattering obstacle using
    the spectral data of the far field operator}, Inverse Problems,
  14(6), (1998),
  1489--1512. \\\url{http://dx.doi.org/10.1088/0266-5611/14/6/009}

\bibitem{kirsch08_factor_method_inver_probl} A. Kirsch, N. Grinberg,
  \textit{The factorization method for inverse problems}, Oxford
  University Press, Oxford,
  (2008). \\\url{https://doi.org/10.1093/acprof:oso/9780199213535.001.0001}

\bibitem{morimoto98_analy} M. Morimoto, \textit{Analytic functionals
  on the sphere}, American Mathematical Society, Providence, RI,
  (1998).

\bibitem{newton62_const_poten_from_phase_shift} R. G. Newton,
  \textit{Construction of potentials from the phase shifts at fixed
    energy}, Journal of Mathematical Physics, 3(1), (1962),
  75--82. \\\url{http://dx.doi.org/10.1063/1.1703790}

\bibitem{NIST}\textit{NIST Digital Library of Mathematical Functions},
  \\\url{http://dlmf.nist.gov/}, Release 1.0.9 of 2014-08-29, online
  companion to \cite{olver10_nist_handb_mathem_funct}.

\bibitem{olver10_nist_handb_mathem_funct} F. W. J. Olver,
  D. W. Lozier, R. F. Boisvert, C. W. Clark (eds.): \textit{NIST
    Handbook of Mathematical Functions}, Cambridge University Press,
  2010, print companion to \cite{NIST}.

\bibitem{paeivaerinta10_inver_scatt_magnet_schroe_operat}
  L. P\"aiv\"arinta, M. Salo, G. Uhlmann, \textit{Inverse scattering
    for the magnetic Schr\"odinger operator}, Journal of Functional
  Analysis, 259(7), (2010),
  1771--1798. \\\url{http://dx.doi.org/10.1016/j.jfa.2010.06.002}

\bibitem{paeivaerinta17_stric_convex_corner_scatt} L. P\"aiv\"arinta,
  M. Salo, E. V. Vesalainen, \textit{Strictly convex corners scatter},
  Revista Matem\'atica Iberoamericana, 33(4), (2017),
  1369--1396. \\\url{http://dx.doi.org/10.4171/rmi/975}

\bibitem{regge59_introd_to_compl_orbit_momen} T. Regge,
  \textit{Introduction to complex orbital momenta}, Il Nuovo Cimento,
  14(5), (1959),
  951--976. \\\url{http://dx.doi.org/10.1007/bf02728177}

\bibitem{salo21_free_bound_method_non_scatt_phenom} M. Salo,
  H. Shahgholian, \textit{Free boundary methods and non-scattering
    phenomena}, arXiv e-prints:2106.15154,
  (2021). \\\url{https://arxiv.org/2106.15154}

\bibitem{sickel99_point_multip_lizor_trieb_spaces} W. Sickel,
  \textit{Pointwise multipliers of Lizorkin-Triebel Spaces}. In, The
  Maz'ya Anniversary Collection, Volume 2: Rostock Conference on
  Functional Analysis, Partial Differential Equations and
  Applications, Birkh\"auser Basel, (1999),
  295--321. \\\url{https://doi.org/10.1007/978-3-0348-8672-7_17}

\bibitem{stein71_introd_fourier_euclid} E. M. Stein, G. Weiss,
  \textit{Introduction to Fourier analysis on Euclidean spaces},
  Princeton University Press, Princeton, N.J., (1971).

\bibitem{stratmann96_box_count_dimen_geomet_finit_klein_group}
  B. Stratmann, M. Urba\'nski, \textit{The box-counting dimension for
    geometrically finite Kleinian groups}, Fundamenta Mathematicae,
  149(1), (1996), 83--93.

\bibitem{szego75_orthog_polyn} G. Szeg\H{o}, \textit{Orthogonal
  Polynomials}, American Mathematical Society, Providence, R.I.,
  (1975).

\bibitem{triebel92_theor_funct_spaces_ii} H. Triebel, \textit{Theory
  of function spaces II}, Springer Basel, (1992).

\bibitem{vesalainen14_rellic_type_theor_unboun_domain}
  E. V. Vesalainen, \textit{Rellich type theorems for unbounded
    domains}, Inverse Problems \& Imaging, 8(3), (2014),
  865--883. \\\url{http://dx.doi.org/10.3934/ipi.2014.8.865}

\bibitem{weck03_approx_by_hergl_wave_funct} N. Weck,
  \textit{Approximation by Herglotz wave functions}, Mathematical
  Methods in the Applied Sciences, 27(2), (2003)
  155--162. \\\url{http://dx.doi.org/10.1002/mma.448}

\end{thebibliography}
\end{document}